\documentclass{amsart}
\usepackage{amsmath, amssymb, amsthm, marginnote, ltxcmds, adjustbox, stmaryrd, graphicx, bbold, extarrows}
\usepackage{esvect}

\usepackage{mathtools, stackengine, changepage, ragged2e}

\usepackage{todonotes}

\usepackage{url}

\usepackage[dvipsnames]{xcolor}
\definecolor{cite}{HTML}{11871E}
\definecolor{url}{HTML}{698996}
\definecolor{link}{HTML}{912F1B}
\definecolor{greencolor}{rgb}{0,0.45,0}
\definecolor{greenbluecolor}{rgb}{0,0.45,0.1}
\definecolor{ucphcolor}{rgb}{0.517,0.016,0.016}
\definecolor{OliveGreen}{rgb}{0,0.6,0}
\usepackage[pdfencoding=unicode, colorlinks=true, linkcolor=ucphcolor, citecolor=greenbluecolor, urlcolor=url]{hyperref}
\usepackage[backend=biber, style=alphabetic, maxnames=10, maxalphanames=3, minalphanames=3]{biblatex}
\usepackage{tikz}
\usetikzlibrary{cd}
\usetikzlibrary{arrows, arrows.meta, positioning, calc}
\tikzcdset{arrow style=tikz, diagrams={>={Straight Barb[scale=0.8]}}}

\tikzstyle{arrow} = [-{Straight Barb[scale=0.8]}, line width=0.2mm]
\tikzset{
math to/.tip={Glyph[glyph math command=rightarrow]},
loop/.tip={Glyph[glyph math command=looparrowleft, swap]},
}

\usepackage{enumitem}
{\end{enumerate}%
}

\usepackage[
	letterpaper,
	twoside=false,
	textheight=22cm,
	textwidth=14.4cm,
	marginparsep=0.75cm,
	marginparwidth=2.5cm,
	heightrounded,
	centering
]{geometry}

\usepackage{lineno}

\usepackage{libertine}
\usepackage{stmaryrd}
\usepackage[T1]{fontenc}

\usepackage[capitalise]{cleveref}
\Crefname{prop}{Proposition}{Propositions}
\Crefname{lem}{Lemma}{Lemmas}
\Crefname{cor}{Corollary}{Corollaries}
\Crefname{thm}{Theorem}{Theorems}
\Crefname{alphThm}{Theorem}{Theorems}

\Crefname{defn}{Definition}{Definitions}
\Crefname{notation}{Notation}{Notations}
\Crefname{cons}{Construction}{Constructions}
\Crefname{rmk}{Remark}{Remarks}
\Crefname{obs}{Observation}{Observations}
\Crefname{trick}{Trick}{Tricks}
\Crefname{warning}{Warning}{Warnings}
\Crefname{conj}{Conjecture}{Conjectures}
\Crefname{assump}{Assumption}{Assumptions}
\Crefname{recollect}{Recollection}{Recollections}
\Crefname{terminology}{Terminology}{Terminologies}

\Crefname{question}{Question}{Questions}
\Crefname{example}{Example}{Examples}

\Crefname{figure}{Figure}{Figures}

\crefformat{equation}{(#2#1#3)}
\crefformat{section}{\S#2#1#3}
\crefmultiformat{section}{\S\S#2#1#3}{and~#2#1#3}{, #2#1#3}{, and~#2#1#3}

\newtheorem{thm}{Theorem}[subsubsection]
\newtheorem{prop}[thm]{Proposition}
\newtheorem{lem}[thm]{Lemma}
\newtheorem{cor}[thm]{Corollary}

\newtheorem{alphThm}{Theorem}

\newcommand{\neutralize}[1]{\expandafter\let\csname c@#1\endcsname\count@}
\makeatother

\newtheorem*{thm*}{Theorem}
\newtheorem*{prop*}{Proposition}
\newtheorem*{lem*}{Lemma}
\newtheorem*{cor*}{Corollary}
  
\newtheorem{alphConj}{Conjecture}



\newtheorem{alphCor}{Corrollary}



\makeatother

\theoremstyle{definition}
\newtheorem*{defn*}{Definition}
\newtheorem{defn}[thm]{Definition}
\newtheorem{cons}[thm]{Construction}
\newtheorem{nota}[thm]{Notation}

\newtheorem{recollect}[thm]{Recollections}

\newtheorem{question}[thm]{Question}

\newtheorem{example}[thm]{Example}

\newtheorem{strategy}[thm]{Strategy}

\newtheorem{rmk}[thm]{Remark}
\newtheorem{obs}[thm]{Observation}

\theoremstyle{definition}

\theoremstyle{remark}

\newcommand{\bbZ}{\mathbb{Z}}

\newcommand{\A}{\mathcal{A}}
\newcommand{\B}{\mathcal{B}}
\newcommand{\sC}{{\mathcal C}}
\newcommand{\D}{{\mathcal D}}

\newcommand{\I}{\mathcal{I}}

\newcommand{\sL}{\mathcal{L}}

\newcommand{\spc}{\mathcal{S}}
\newcommand{\cat}{\mathrm{Cat}}
\newcommand{\presentable}{\mathrm{Pr}}
\newcommand{\spectra}{\mathrm{Sp}}

\newcommand{\op}{^{\mathrm{op}}}

\DeclareMathOperator{\map}{Map}
\DeclareMathOperator{\func}{Fun}
\DeclareMathOperator{\mapsp}{map}
\DeclareMathOperator{\presheaf}{Psh}

\newcommand{\aut}{\mathrm{Aut}}

\newcommand{\calg}{\mathrm{CAlg}}
\newcommand{\module}{\mathrm{Mod}}
\newcommand{\Pic}{\mathcal{P}\mathrm{ic}}
\newcommand{\sphere}{\mathbb{S}}
\newcommand{\unit}{\mathbb{1}}
\newcommand{\id}{\mathrm{id}}

\newcommand{\stmodSmall}{\mathrm{stmod}}

\DeclareMathOperator{\fib}{fib}
\DeclareMathOperator{\cofib}{cofib}

\DeclareMathOperator{\eval}{ev}
\DeclareMathOperator{\coev}{coev}

\DeclareMathOperator{\res}{Res}
\DeclareMathOperator{\infl}{infl}

\newcommand{\conn}{\mathrm{conn}}

\newcommand{\family}{\mathcal{F}}
\newcommand{\orbit}{\mathcal{O}}

\newcommand{\proper}{\mathcal{P}}

\newcommand{\free}{\mathrm{free}}
\newcommand{\isov}{\mathrm{isov}}

\newcommand{\hrep}{\mathcal{V}}
\newcommand{\Man}{\mathcal{M}\mathrm{an}}
\newcommand{\PD}{\mathrm{PD}}
\newcommand{\semifree}{\mathrm{sf}}
\newcommand{\PDsf}[1]{\PD^{\semifree}_{#1}}
\newcommand{\PDsfpair}[1]{\PD^{\partial, \semifree}_{#1}}
\newcommand{\PDsfisov}[1]{\PD^{\semifree}_{{#1},\mathrm{isov}}}
\newcommand{\PDsfisovpair}[1]{\PD^{\partial, \semifree}_{{#1},\mathrm{isov}}}
\newcommand{\isovstr}{\mathrm{Isov}}
\newcommand{\hrepfree}[1]{\hrep_{#1}^\mathrm{free}}
\newcommand{\spectraisotropysphere}[3]{\spc_{#1, *}^{#2}[{#3}^{-1}]}

\newcommand{\equivdualisingspectrum}[2]{D_{#1, #2}}

\newcommand{\join}{\star}

\newcommand{\udl}[1]{\underline{{#1}}}

\def\colim{\qopname\relax m{colim}}

\newcommand{\arrdisp}{0.33ex}
\newcommand{\arrdisplacementsp}{0.72ex}

\newcommand{\ardis}{\ar@<\arrdisp>}
\newcommand{\ardissp}{\ar@<\arrdisplacementsp>}

\newcommand*\cocolon{%
        \nobreak
        \mskip6mu plus1mu
        \mathpunct{}%
        \nonscript
        \mkern-\thinmuskip
        {:}%
        \mskip2mu
        \relax
}

\title{Semifree Isovariant Poincar\'e Spaces and the Gap Condition}
\author{DOMINIK KIRSTEIN}
\email{kirstein@math.lmu.de}
\address{Mathematisches Institut, Ludwig-Maximilians-Universität München, Theresienstraße 39, 80333 Munich, Germany}
\author{CHRISTIAN KREMER}
\email{kremer@mpim-bonn.mpg.de}
\address{Max Planck Institute for Mathematics, Vivatsgasse 5, 53111 Bonn, Germany}

\date{\today}

\begin{document}
\subjclass[2020]{
55P91, %Equivariant homotopy theory in algebraic topology
57P10, %Poincaré duality spaces
57R40, %Embeddings in differential topology
57S17 %Finite transformation groups
}

\begin{abstract}
    We introduce the notion of a semifree isovariant $G$-Poincar\'e space, a homotopical notion interpolating between semifree closed smooth $G$-manifolds and the equivariant Poincar\'e spaces of \cite{HKK_PD}. 
    It carries the additional structure of an equivariant Poincar\'e embedding of the fixed points of a semifree $G$-Poincar\'e space. 
    Under suitable gap conditions on the codimension, we show that the space of isovariant structures on a semifree $G$-Poincar\'e space for a periodic finite group $G$ is highly connected, giving a useful construction tool for manifold structures on equivariant Poincar\'e spaces.
\end{abstract}
\maketitle

\tableofcontents

\section{Introduction}

\counterwithin{thm}{section}

The study and classification of group actions on closed manifolds has been a cornerstone of geometric topology throughout the development of the field. As a central example, Madsen--Thomas--Wall completely characterised those finite groups which admit a free topological action on sphere \cite{Madsen_Thomas_Wall}. Equivalently, they characterised all finite groups that occur as fundamental groups of closed manifolds whose universal cover is the sphere. This seminal work strongly relies on work of Swan \cite{Swan_periodic}, who solved the homotopical counterpart to this question, namely asking which finite groups can appear as fundamental groups of a Poincar\'e space which is covered by the sphere. We want to stress that the full program was solved by splitting it in two -- a homotopical part, that was studied by Swan by means of unstable homotopy theory, and a geometric part, for which Wall's non-simply connected surgery theory was crucial.

A substantial amount of progress has been made on the construction and classification of non-free actions as well. However, a simple procedure, such as passing to the quotient and solving a problem in nonequivariant manifold topology instead, is no longer available. 
While there has been a considerable amount of work on equivariant surgery, the homotopical side has, until recently, only sparsely been studied. 
This motivated the authors to extensively study the notion of \textit{$G$-equivariant Poincar\'e spaces} \cite{HKK_PD,HKK_Nielsen,BHKK_PDP} to lay solid foundations for the classification and study on nonfree group actions on manifolds.
For the main results of this article, we focus on \textit{semifree group actions}. Here, for a finite group $G$, a $G$-space $X$ is \textit{semifree} if for each subgroup $e \neq H \leq G$, the map $X^G \rightarrow X^H$ is an equivalence.
Equivalently, it can be built from cells of free isotropy type $G/e$ or of fixed  isotropy type $G/G$. A \textit{semifree $G$-Poincar\'e space} is a $G$-Poincar\'e space in the sense of \cite{HKK_PD} which is also semifree.

To study the moduli space of semifree $G$-manifolds $\Man_G^\mathrm{sf}$, this article considers a factorisation
\[ \Man_G^\mathrm{sf} \rightarrow \PDsfisov{G} \rightarrow \PDsf{G}. \]
The middle space is the moduli space of \textit{semifree isovariant $G$-Poincar\'e spaces}. It is motivated by the observation that automorphisms of $G$-manifolds preserve more homotopical structure than merely equivariant maps. Recall that an equivariant map of topological $G$-spaces $f \colon X \rightarrow Y$ is \textit{isovariant} if it preserves isotropy groups, i.e. $G_x = G_{f(x)}$ for each $x \in X$. Automorphisms of $G$-manifolds are isovariant maps, so we conclude that the isovariant homotopy type of $G$-manifolds, as a natural piece of structure on their equivariant homotopy type, should be taken into account in their study. We give a definition of isovariant structures adapted to our needs, and compare with Yeakel's recent work on the homotopy theory of isovariant spaces \cite{Yaekel} in \cref{rmk:yeakel}.
The following is the main result of this article.

\begin{alphThm}[{\cref{thm:main_thm}}]\label{thmintro:main_thm}
    Let $X$ be a semifree $G$-Poincar\'e space and $G$ a periodic finite group.
    Consider $k \ge -1$ such that for each component of $X^G$ and the corresponding component of $X^e$ containing it we have
    \begin{enumerate}
        \item $\dim(X^G) + 3 \le \dim(X^e)$;
        \item $k \le \dim(X^e) - 2 \dim(X^G) - 3 $.
    \end{enumerate}
    Then the space $\isovstr_G(X) = \PDsfisov{G} \times_{\PDsf{G}} \{X\}$ of isovariant structures on $X$ is $k$-connected.
\end{alphThm}

For $k=-1$, $k$-connected means nonempty.
To clarify the inequalities occuring in the theorem, let us recall that a Poincar\'e space has a \textit{dimension}, a natural number valued function on its components. The inequalities in \cref{thm:main_thm} should be read as inequalities on the dimension function of $X^G$ and that of $X^e$ restricted to $X^G$ along the inclusion. The second condition is usually referred to as a \textit{gap hypothesis}.

Next, we give our definition of an isovariant structure on a semifree $G$-Poincar\'e space, before explaining how an semifree smooth closed $G$-manifold gives rise to such a structure on its underlying $G$-Poincar\'e space.

\begin{defn}
    \label{def:isov_structure}
    Given a semifree $G$-Poincar\'e space $X$, an isovariant structure on $X$ is a pushout of compact $G$-spaces
\begin{equation}\label{diag:isovariation_pushout}
    \begin{tikzcd}
        \partial C \ar[d, "p"] \ar[r] \ar[dr, phantom, very near end, "\ulcorner"] 
        & C \ar[d]
        \\
        X^G \ar[r]
        & X,
\end{tikzcd}
\end{equation}
where the lower horizontal morphism is inclusion of the fixed points of $X$, subject to the following conditions.
\begin{enumerate}
    \item The map $p$ is an equivariant spherical fibration.
    \item The $G$-action on both $C$ and $\partial C$ is free.
    \item The pair $(C^e, \partial C^e)$ is a nonequivariant Poincar\'e pair.
\end{enumerate}
\end{defn}
Suppose that the semifree $G$-Poincar\'e space $X$ admits an isovariant structure and that the codimension $\dim(X^e)-\dim(X^G)$ is at least 1, i.e., $X^G \to X^e$ is not just an inclusion of components.
Then the finite group $G$ freely acts on the spheres arising as the fibres of $p$, which forces it to be periodic.
This explains why the assumption that $G$ is periodic in \cref{thmintro:main_thm} is necessary. 

In practice, an advantage of isovariant $G$-Poincar\'e spaces over equivariant $G$-Poincar\'e spaces is that the decomposition into the free part and the fixed part allows one to apply surgery theoretic techniques to both parts separately.
Let us also mention that \cref{thm:main_thm} gives the best currently available method to classify a good amount of semifree equivariant Poincar\'e spaces, because the decomposition into a free part and a fixed part allows to phrase it in terms of classifications of nonequivariant Poincar\'e duality spaces and pairs.
We proceed by giving two immediate geometric applications to the study of group actions on manifolds.

\subsection*{Application: The Browder--Straus theorem}
\setcounter{thm}{1}

As one application of \cref{thm:main_thm} we show that it recovers a classical theorem on isovariant maps between smooth closed $G$-manifolds, under a slightly stronger gap hypothesis.
Let $M$ be a closed semifree smooth $G$-manifold. It has an underlying semifree isovariant $G$-Poincar\'e space, described as follows.

\begin{cons}
    The inclusion of the fixed points $\epsilon \colon M^{G} \rightarrow M$ is a smooth embedding. We write $\nu$ for its normal bundle, and $S\nu$ for the unit sphere bundle in that normal bundle, after a choice of an equivariant Riemannian metric, and $D\nu$ for the associated disc bundle. A choice of an appropriate equivariant tubular neighborhood defines an embedding $D\nu \subset M$, restricting to the identity on $M^{G}$. On underlying $G$-spaces, we get a pushout in $\spc_{G}^{\omega}$ as follows.
\begin{equation}
    \begin{tikzcd}
        S\nu \ar[r] \ar[d] & M \setminus M^{C_p} \ar[d] \\
        M^{G} \simeq D\nu \ar[r] & M
    \end{tikzcd}
\end{equation}
This square defines a semifree isovariant structure on the $G$-Poincar\'e space underlying $M$.
\end{cons}

Using a comparison to the homotopy theory of isovariant spaces developed by Yeakel and Klang--Yeakel that we give in \cref{rmk:yeakel}, our result recovers the following version of the Browder--Straus-theorem, see \cite{Schultz_Gap}. 

\begin{cor}
    Let $G$ be a periodic group and let $M$ and $N$ be semifree closed smooth $G$-manifolds. Assume that $\dim M^e - \dim M^G \geq 3$. Then
    \begin{enumerate}
        \item if $2\dim M^G + 3 \leq \dim M^e$, any $G$-equivariant homotopy equivalence $f \colon M \rightarrow N$ may be lifted to an isovariant one;
        \item if $2\dim M^G + 4 \leq \dim M^e$, any two $G$-isovariant homotopy equivalences $f \colon M \rightarrow N$ which are equivariantly homotopic, are isovariantly homotopic.
    \end{enumerate}
\end{cor}

Note that the classical Browder--Straus theorem has a slightly better range only assuming $2\dim M^G + 2 \leq \dim M^e$. Our approach of course applies to more general merely equivariant maps of isovariant $G$-Poincar\'e spaces, and losing a dimension when passing from manifolds to Poincar\'e spaces is not uncommon, see \cite[p. 2]{Klein2002b}.

\subsection*{Application: Isovariance structures in the Nielsen realisation problem}
\setcounter{thm}{4}

One of the main motivations for this article is the Nielsen realisation problem. We say that a \textit{homotopical $G$-action} on a manifold is a map of $E_1$-groups $G \rightarrow \mathrm{hAut}(M)$. The high-dimensonal Nielsen realisation problem for aspherical manifolds is about rigidifying such actions.

\begin{question}[The Nielsen realisation problem, Borel version]
    \label{quest:nielsen}
    If $G$ is a finite group and $M$ is a closed aspherical manifold with a homotopical $G$-action, when is there a $G$-action on $M$ by homeomorphisms giving rise to the $G$-homotopy type $\mathrm{Bor}(M)$?
\end{question}

Here the $G$-homotopy type $\mathrm{Bor}(M)$ is the obtained by putting $\mathrm{Bor}(M)^H = M^{hH}$, using the homotopical $G$-action. See \cite{kremer2025borelactionsnonpositivelycurved} for context and the relation to other formulations.
Recent strategies to answer \cref{quest:nielsen} have relied on constructing the structure of a $G$-isovariant Poincar\'e space on $\mathrm{Bor}(M)$ first, see \cite{lueck2022brown, davis2024nielsen}. Our result is the first such which shows the existence of isovariant Poincar\'e structures in the case where $\mathrm{Bor}(M)$ is not pseudofree, i.e., the fixed points are not discrete. 

\begin{cor}
    \label{cor:nielsen}
    In the situation of \cref{quest:nielsen}, assume that  $G$ is periodic and $\mathrm{Bor}(M)$ $G$ is a semifree $G$-Poincar\'e space. Then if $\dim M - \dim M^{hG} \geq 3$ and $2 \dim M^{hG} + 2 \leq \dim M$, $\mathrm{Bor}(M)$ admits the structure of a semifree isovariant $G$-Poincar\'e space.
\end{cor}

In future work, we plan to use \cref{cor:nielsen} combined with the main result of \cite{HKK_Nielsen} to answer \cref{quest:nielsen} in a much broader class of examples than is currently known. The importance of \cref{cor:nielsen} is that using the decomposition provided isovariant structure on $\mathrm{Bor}(M)$, one is put in a good position to construct manifolds with boundaries for the pieces of the decomposition, and glue them together to build a manifold with a $G$-action.

\subsection*{Proof strategy and organisation of the article}

The proof strategy for \cref{thmintro:main_thm} consists of two steps.
One first observes that, given a $G$-Poincar\'e space $X$, the spherical fibration $p$ in \cref{diag:isovariation_pushout} always exists stably as the "stable equivariant normal bundle" of $X^G$ in $X$, and may be built from the dualising systems of $X^G$ and $X$.
The goal of the first step is to destabilise this stable normal bundle $\nu \colon X^G \to \Pic(\spectra_G)$ along the stabilisation map $\Sigma^\infty_J \colon \hrepfree{G} \to \Pic(\spectra_G)$ to an equivariant spherical fibration of the correct dimension.
Here, $\hrepfree{G}$ denotes the moduli space of tom Dieck's free generalised $G$-homotopy representations. To study it, we build a custom-made category of semifree $G$-spectra when $G$ is a periodic finite group, which we believe to be of some independent interest.
In the second step, we build the complement $C$ in \cref{diag:isovariation_pushout} by obstruction theory, by lifting the relative cells of the pair $(X, X^G)$ along $p$.
Naively, this only works up to half the dimension of $X$.
We employ Klein's nonequivariant existence result for Poincar\'e embeddings \cite[Theorem A]{Klein2002b} to actually lift all of those relative cells along $p$.

In the first part of this article \cref{sec:setup}, we recall the necessary background on Poincar\'e embeddings and equivariant Poincar\'e spaces needed in this article and introduce semifree isovariant $G$-Poincar\'e spaces in \cref{subsec:isovariant_semifree_G_PD}.
The destabilisation part of the proof strategy will be completed in \cref{sec:destabilisations}, and the obstruction theoretic part appears in \cref{sec:existence_isovariant_structures}.
The construction of the category of semifree $G$-spectra is deferred to \cref{sec:spectra_specified_isotropy}.

\subsection*{Notations and conventions}

We freely use the language and theory of $\infty$-categories as developed by Joyal, Lurie and many others. The term \textit{category} will refer to an $\infty$-category. We write $\spc$ for the (large) category of spaces, and $\spectra$ for the (large) category of spectra. If $G$ is a finite group, we write $\spc_G$ for the category of $G$-spaces, modelled as the category of $\spc$-valued presheaves on the orbit category of $G$, and we denote the category of genuine $G$-spectra by $\spectra_G$. We tried to make this article accessible without detailed knowledge of parametrised category theory, although it will appear in remarks that we deem helpful for the knowledgeable reader.

\subsection*{Acknowledgements}

We wholeheartedly thank Kaif Hillman for countless helpful conversations about equivariant Poincar\'e duality, and for introducing us to genuine equivariant homotopy theory in the first place. We also thank our advisor Wolfgang L\"uck for encouragement, discussions and good pointers to the literature when they were needed.
DK thanks the Ludwig-Maximilians-Universität München for their conducive working environments.
Both authors would like to thank the Max Planck Institute for Mathematics (MPIM) in Bonn for its hospitality.
\section{The setup}\label{sec:setup}

\counterwithin{thm}{subsection}

We begin by recalling some notions and constructions on Poincar\'e pairs and embeddings as well as equivariant Poincar\'e spaces that we use throughout the article in \cref{sec:setup_pd_pairs,sec:setup_g_pd}
In \cref{subsec:isovariant_semifree_G_PD} we introduce semifree $G$-Poincar\'e spaces.

\subsection{Poincar\'e pairs and embeddings} \label{sec:setup_pd_pairs}

According to a deep insight by Klein \cite{Klein_2001}, a compact space $X \in \spc^\omega$ comes with a \textit{dualising system} $D_X \in \spectra^X$ of spectra, uniquely characterised by the equivalence
\begin{equation*}
    X_* \simeq X_!(- \otimes D_X)
\end{equation*}
under the Morita-theoretic classification
\begin{equation}\label{eq:morita_classification}
    \spectra^X \xrightarrow{\simeq} \func^L(\spectra^X, \spectra), \quad E \mapsto X_!(- \otimes E)
\end{equation}
of colimit preserving functors.
Here $X_!, X_* \colon \spectra^X \to \spectra$ denote the colimit and limit functors, the left and right adjoints to the restriction functor $X^* \colon \spectra \to \spectra^X$.
The compact space $X$ is called a \textit{Poincar\'e space} if $D_X$ is pointwise invertible.
In classical terms, $D_X$ is the fibrewise Thom spectrum of the Spivak normal fibration of $X$.

There are also relative versions of this notion:
For a map $i \colon \partial X \to X$ of compact spaces we call
\begin{equation}\label{eq:fiber_sequence_rel_dualising_spectrum}
    D_{(X, \partial X)} = \fib(D_X \to i_! D_{\partial X})
\end{equation}
the \textit{relative dualising spectrum} of the pair $(X, \partial X)$,
where the map $D_X \to i_! D_{\partial X}$ corresponds to the map $X_* \to X_* i_*i^* \simeq \partial X_* i^*$ induced by the adjunction unit $\id \to i_* i^*$ under \cref{eq:morita_classification}.
Here, $i_!, i_* \colon \spectra^{\partial X} \to \spectra^X$ denote the left and right Kan extension functors, which are left and right adjoint to the restriction functor $i^* \colon \spectra^X \to \spectra^{\partial X}$, respectively.
$(X, \partial X)$ is called a \textit{Poincar\'e pair} if $D_{(X, \partial X)}$ is pointwise invertible and the map
\begin{equation*}
    \Omega D_{\partial X} \xrightarrow{} \Omega i^* i_! D_{\partial X} \to i^* D_{(X, \partial X)}
\end{equation*}
induced by the adjunction unit $\id \to i^* i_!$ and the connecting map of the fibre sequence \cref{eq:fiber_sequence_rel_dualising_spectrum} is an equivalence.
We will also need the notion of a \textit{Poincar\'e triad} $(X; X_0, X_1; X_{01})$, which is a commutative square of spaces
\begin{equation*}
\begin{tikzcd}
    X_{01} \ar[r] \ar[d] 
    & X_0 \ar[d] 
    \\
    X_1 \ar[r]
    & X
\end{tikzcd}
\end{equation*}
such that $(X_0, X_{01})$, $(X_1, X_{01})$ and $(X, X_0 \amalg_{X_{01}} X_1)$ are Poincar\'e pairs.

Let us recall the following basic facts on Poincar\'e pairs that we  use throughout the article.
These results are well known in the classical formulation via fundamental classes.
A proof in the formulation through parametrised spectra can be found in \cite{BHKK_PDP}.
\begin{lem}\label{lem:basic_facts_pd_pairs}
    \begin{enumerate}
        \item (Pushouts) Consider a pushout square of compact spaces
        \begin{equation*}
        \begin{tikzcd}
            X_{01} \ar[r] \ar[d] \ar[dr, phantom, very near end, "\ulcorner"]
            & X_0 \ar[d, "i_0"] 
            \\
            X_1 \ar[r, "i_1"]
            & X.
        \end{tikzcd}
        \end{equation*}
        If $(X_0, X_{01})$ and $(X_1, X_{01})$ are Poincar\'e pairs, then $X$ is a Poincar\'e space and the map $D_{(X_0, X_{01})} \to D_{X_0} \to i_0^* (i_0)_! D_{X_0} \to i_0^* D_X$ is an equivalence.
        Conversely, if the map $X_1 \to X$ admits a retraction on fundamental groupoids, if $X$ and $(X_, X_{01})$ are Poincar\'e spaces and if the map $D_{(X_0, X_{01})} \to i_0^* D_X$ is an equivalence, then $(X_1, X_{01})$ is a Poincar\'e pair.
        \item (Fibrations) Consider a map $p \colon X \to Y$ of compact spaces such that all fibres of $p$ are compact.
        Then there is an equivalence $D_X \simeq D_p \otimes p^* D_Y$ for a parametrised spectrum $D_p \in \spectra^X$.
        It comes together with an identification $i_y^* D_p \simeq D_{p^{-1}(y)}$ for all $y \in Y$, where $i_y \colon p^{-1}(y) \to X$ denotes the inclusion of the fibre.
        \item (Spheres) If $p \colon X \to Y$ is a spherical fibration over a Poincar\'e space $Y$, then $(Y, X)$ is a Poincar\'e pair and one has $D_{(Y, X)} \simeq p^* D_Y \otimes (\Sigma^\infty_J p)^{-1}$, where $\Sigma^\infty_J$ denotes the fibrewise join stabilisation of $p$.
        \item (Relative fibrations)
        Consider a map $(p, \partial p) \colon (E, \partial E) \to B$ of spaces.
        Assume that $B$ is the total space of a Poincar\'e pair $(B, \partial B)$ and that all fibres $(F, \partial F)$ of $(p, \partial p)$ are compact.
        Then $(E; \partial E, E \times_B \partial B; \partial E \times_{B} \partial B)$ is a Poincar\'e triad if and only if all fibres $(F, \partial F)$ are Poincar\'e pairs.
    \end{enumerate}
\end{lem}

Our proof requires the following existence result for Poincar\'e embeddings in the nonequivariant case from \cite[Theorem A]{Klein2002b}.
\begin{defn}\label{def:relative_poincare_embedding}
    Consider a map $(f, \partial f) \colon (L, \partial L) \to (X, \partial X)$ of Poincar\'e pairs.
    A \textit{Poincar\'e embedding structure} on $(f, \partial f)$ is a pushout of pairs
    \begin{equation}\label{diag:pushout_equivariant_pd_embedding}
    \begin{tikzcd}
        (\partial_0 C, \partial_{01} C) \ar[r] \ar[d, "{(\nu, \partial \nu)}"'] \ar[dr, phantom, very near end, "\ulcorner"]
        & (C, \partial_1 C) \ar[d]
        \\
        (L, \partial L) \ar[r, "{(f, \partial f)}"']
        & (X, \partial X)
    \end{tikzcd}
    \end{equation}
    such that $(C; \partial_0 C, \partial_1 C; \partial_{01} C)$ is a Poincar\'e triad and $(\nu, \partial \nu) \colon (\partial_0 C, \partial_{01} C) \to (L, \partial L)$ is a relative spherical fibration.
\end{defn}

\begin{thm}[Klein, {\cite[Theorem A]{Klein2002b}}]\label{thm:nonequivariant_existence_embeddings}
    Consider a map of Poincar\'e pairs $f \colon (L,\partial L) \rightarrow (X,\partial X)$, where $L$ and $\partial L$ are finite spaces.
    Suppose that we are given a Poincar\'e embedding structure on $\partial f \colon \partial L \to \partial X$ and that the following conditions are satisfied:
    \begin{enumerate}
        \item each component of the pair $(L, \partial L)$ has dimension at most $k$;
        \item each component of $(X, \partial X)$ has dimension at least $d$;
        \item the map $f \colon L \to X$ is $r$-connected;
        \item $k \le d-3$ and $r \ge 2k-d+2$.
    \end{enumerate}
    Then there exists a relative Poincar\'e embedding structure on $f$, restricting to the given one on the boundary $\partial f$.
\end{thm}

\subsection{Equivariant Poincar\'e duality}\label{sec:setup_g_pd}

Equivariant Poincar\'e duality is a notion developed by the authors in \cite{HKK_PD}, and further studied in \cite{HKK_Nielsen,BHKK_PDP}, to express the (co)homological behaviour of smooth closed $G$-manifolds.
For the readers convenience, we give a precise recollection of the main facts of that theory that are relevant to the rest of the article. The most important concept for us is the \textit{equivariant dualising system} of a compact $G$-space, which roughly collects all the dualising spectra of the various fixed points with their compatibilities and equivariance. For our purposes it suffices to know that for each compact $G$-space there is a local system of genuine $G$-spectra $\equivdualisingspectrum{X}{G} \colon X^{G} \rightarrow \spectra_G$, which enjoys the following two compatibilities with the nonequivariant dualising spectra of $X^G$ and $X^e$.
\begin{thm}\label{thm:basic_properties_equivariant_dualising_spectrum}
\begin{enumerate}
    \item There is a commutative square
    \begin{equation*}
    \begin{tikzcd}
        X^G \ar[d] \ar[r, "\equivdualisingspectrum{X}{G}"]
        & \spectra_G \ar[d, "\res"]
        \\
        X^e \ar[r, "D_{X^e}"]
        & \spectra.
    \end{tikzcd}
    \end{equation*}
    \item The composite
    \begin{equation*}
        X^G \xrightarrow{\equivdualisingspectrum{X}{G}} \spectra_G \xrightarrow{\Phi^G} \spectra
    \end{equation*}
    identifies with the nonequivariant dualising spectrum $D_{X^G}$.
\end{enumerate}
\end{thm}
\begin{proof}
    This is  \cite[Proposition 4.3.1]{HKK_PD} and \cite[Theorem 4.2.9]{HKK_PD}.
\end{proof}
$X$ is called a \textit{$G$-Poincar\'e space} if this equivariant dualising spectrum (and also the dualising spectra $\equivdualisingspectrum{X}{H} \colon X^H \to \spectra_H$ for all intermediate sugroups $H \le G$) is invertible. Examples are closed smooth manifolds with a smooth action of the group $G$ \cite[Prop. 4.4.2.]{HKK_PD}.

We are mainly interested in \textit{semifree $G$-spaces}.
This means that the map $X^G \to X^H$ is an equivalence for all subgroups $e \neq H \le G$, or equivalently, that $X$ has the homotopy type of a $G$-CW complex with cells of isotropy type $G/G$ and $G/e$ - each point either lies in a free orbit or is fixed by the group action. A \textit{semifree $G$-Poincar\'e space} is a semifree compact $G$-space, which is also a $G$-Poincar\'e space. The \textit{moduli space of semifree $G$-Poincar\'e spaces} will be denoted by $\PDsf{G}$, the full subgroupoid of $\spc_G^{\simeq}$ on all semifree $G$-Poincar\'e spaces.
The cruicial property of the equivariant dualising spectrum for semifree $G$-spaces that we use in this article is the following:
\begin{thm}\label{lem:dualising_singular_part_trivial}
    Let $X$ be a semifree compact $G$-space. Then the following two composites
    \begin{equation*}
        X^G \xrightarrow{\equivdualisingspectrum{X}{G}} \spectra_G \to \spectra_G/e
    \end{equation*}
    and 
    \begin{equation*}
        X^G \xrightarrow{D_{X^G}} \spectra \xrightarrow{\infl} \spectra_G \xrightarrow{} \spectra_G/e
    \end{equation*}
    are equivalent.
\end{thm}
\begin{proof}
    This is a special case of \cite[Thm. 4.2.7.]{HKK_PD} for the trivial family $\family = \{ e \}$, using that the singular part $X^{>1}$ agrees with $X^G$ in the semifree case.
\end{proof}

\subsection{Semifree isovariant $G$-Poincar\'e spaces}
\label{subsec:isovariant_semifree_G_PD}

We have introduced semifree $G$-Poincar\'e spaces in the last section. Note that if a smooth $G$-action on a closed smooth $G$-manifold $M$ is semifree in the sense that each isotropy group is either trivial or all of $G$, then the underlying $G$-homotopy type of $M$ is a semifree $G$-Poincar\'e space. However, in this geometric setting we observe that the underlying $G$-homotopy type of $M$ actually comes with a refined structure in the shape of a decomposition.

The fixed points $M^G$ are a smooth $G$-submanifold of $M$.
The normal bundle $\nu$ inherits a $G$-action whose fibre over a fixed point in $M^G$ is a free $G$-representation.
We can recover $M$ up to $G$-homotopy equivalence by the pushout
\begin{equation*}
\begin{tikzcd}
    S(\nu) \ar[d] \ar[r] \ar[dr, phantom, very near end, "\ulcorner"]
    & M \setminus D(\nu) \ar[d]
    \\
    M^G \ar[r]
    & M,
\end{tikzcd}
\end{equation*}
where $D(\nu) \subseteq M$ denotes an equivariant tubular neighbourhood of $M^G$, the disk bundle of $\nu$, and $S(\nu)$ is its boundary.
The pair $(M \setminus D(\nu), S(\nu))$ is a free $G$-manifold with boundary and the projection $S(\nu) \to M^G$ is an equivariant fibre bundle with fibres given by free $G$-spheres. Next, we aim at capturing this decomposition in a homotopical fashion, which leads to the concept of a semifree isovariant $G$-Poincar\'e space.

To give the homotopical analogue of the sphere normal bundle of the fixed point set, for example, we have to replace the unit spheres in the normal representation $\nu$ by a homotopical analogue, which is provided by tom Dieck's generalised homotopy representations.
\begin{defn}\label{def:generalised_hrep}
    A \textit{generalised $G$-homotopy representation} is a compact $G$-space $V \in \spc_G^\omega$ such that for all subgroups $H \le G$ there is a number $n(H) \ge -1$ and an equivalence $V^H \simeq S^{n(H)}$.
\end{defn}

In later sections, we study generalised $G$-homotopy representations and their relation to invertible $G$-spectra via a suitable process of stabilisation. Now $G$-homotopy representations are used to define the notion of an equivariant spherical fibration.

\begin{defn}
    \label{def:equiv_spherical_fibration}
    An \textit{equivariant spherical fibration} is a map $p \colon X \to Y$ of $G$-spaces such that for each subgroup $H \le G$ and each point $y \in Y^H$ the fibre $p^{-1}(y) \in \spc_H$ is a generalised $H$-homotopy represenation. 
    A \textit{relative equivariant spherical fibration} is a map $(p, p') \colon (X, X') \to (Y, Y')$ of pairs of $G$-spaces such that $p$ is an equivariant spherical fibration and $X' \to X \times_{Y} Y'$ is an equivalence.
\end{defn}

\begin{rmk}
    In terms of parametrised homotopy theory, if a map $p \colon X \rightarrow Y \in \spc_G$ is classified by a functor $t_p \colon Y \to \udl{\spc}^{\simeq}$ to the moduli $G$-space of $G$-spaces, it is an equivariant spherical fibration if $t_p$ lands in the sub $G$-space of generalised $G$-homotopy representations.
\end{rmk}

We have introduced the necessary terminology to introduce the main concept of this article, the concept of an isovariant structure on a semifree $G$-Poincar\'e space $X \in \PDsf{G}$, mirroring the decomposition of a semifree smooth closed $G$-manifold constructed above. 

\begin{defn}\label{def:isovariant_structure}
    For $X \in \PDsf{G}$, an \textit{isovariant structure on $X$} is a pushout
    \begin{equation}\label{diag:isovariant_structure}
        \begin{tikzcd}
            \partial C \ar[r] \ar[d, "p"] & C \ar[d] \\
            X^{G} \ar[r] & X
        \end{tikzcd}
    \end{equation}
    satisfying the following conditions:
    \begin{enumerate}
        \item $p \colon \partial C \to X^G$ is an equivariant spherical fibration;
        \item $C$ and $\partial C$ are free $G$-spaces;
        \item $(C^e, \partial C^e)$ is a Poincar\'e pair.
    \end{enumerate}    
    Denote by $\PDsfisov{G} \subseteq \func([1]^2, \spc_G)^\simeq$ the full subgroupoid consisting of pushout squares which have a semifree $G$-Poincar\'e space as bottom right corner and provide an isovariant structure on it.

    We will also need the following relative version.
    Consider a $G$-Poincar\'e pair $(X, \partial X)$ and assume that $X$ and $\partial X$ are both semifree.
    Then an isovariant structure on $(X, \partial X)$ is a pushout of pairs of $G$-spaces
    \begin{equation}\label{diag:isovariant_pair}
    \begin{tikzcd}
        (\partial_0 C, \partial_{01} C) \ar[r] \ar[d, "{(p, \partial p)}"'] \ar[dr, phantom, very near end, "\ulcorner"]
        & (C, \partial_1 C) \ar[d]
        \\
        (X^{C_p}, \partial X^{C_p}) \ar[r]
        & (X, \partial X)
    \end{tikzcd}
    \end{equation}
    satisfying the following conditions:
    \begin{enumerate}
        \item $(p, \partial p) \colon (\partial_0 C, \partial_{01} C) \to (X^{C_p}, \partial X^{C_p})$ is a relative equivariant spherical fibration;
        \item $C$, $\partial_0 C$, $\partial_1 C$ and $\partial_{01} C$ are free $G$-spaces;
        \item $(C^e; \partial_0 C^e, \partial_1 C^e; \partial_{01} C^e)$ is a Poincar\'e triad.
    \end{enumerate}
    We can define the space of semifree isovariant $G$-Poincar\'e pairs as the full subspace $\PDsfisovpair{G} \subseteq \func([1]^3, \spc_G)^\simeq$ of those cubes whose front and back face are pushouts and which provide a semifree structure on the bottom right leg.
\end{defn}

For those familiar with the notion of equivariant Poincar\'e spaces, an isovariant structure on the semifree $G$-Poincar\'e space $X$ is the same as an equivariant Poincar\'e embedding of the fixed points $X^G \to X$.

\begin{rmk}
    The reader might wonder about the ad-hoc nature of the above definition, and why we do not require that other subdiagrams are Poincar\'e triads, for example. Aside from being natural from a geometric viewpoint, the framework of Poincar\'e duality for category pairs from \cite{BHKK_PDP} gives us a way to check that have put the ``right" conditions. For this remark, we will freely use the language of that article. Note that the cube in \cref{def:isovariant_structure} is determined by the subdiagram
    \begin{equation*}
        \begin{tikzcd}
            \partial_0 C \ar[d] & \ar[l] \partial_{01} C \ar[r] \ar[d] &  \partial_1 C \ar[d] \\
             X^{C_p} & \ar[r] \partial X^{C_p} \ar[l] & C
        \end{tikzcd}
    \end{equation*}
    and we may take the total $G$-category of the unstraightening of that diagram relative to the subcategory determined by the upper span $\partial_0 C\leftarrow \partial_{01} C \rightarrow \partial_1 C$  to get a $G$-category pair $(\mathcal{X}, \partial \mathcal{X})$. According to the local-to-global-principle, we see that this pair is a $G$-Poincar\'e duality pair if and only if both the left square and the right square are $G$-Poincar\'e triads. For the right square, our condition predicts that the left square is a Poincar\'e triad, which is equivalent to it being a $G$-Poincar\'e triad since the action is free. For the left square it follows from $(p,\partial p)$ being a relative equivariant spherical fibration.
\end{rmk}

\begin{rmk}\label{rem:equivalent_conditions_def_isovariant_structure}
    Under codimension assumptions, condition (3) in \cref{def:isov_structure} is sometimes easier to check.
    If for each component of $X^G$ and the corresponding component of $X^e$ containing it one has $\dim(X^e) - \dim(X^G) \ge 3$, then condition (3) is equivalent to asking whether the composite $D_{(X^G, \partial C)} \to D_{X^G} \to i^* i_! D_{X_G} \to i^* D_X$ is an equivalence: The map $\partial C \to X^G$ is 2-connected and the claim then follows from Wall's subtraction result \cite[Theorem 2.1 (ii)]{WallPD}. Similar remarks hold in the relative version.
\end{rmk}

\begin{rmk}[Isovariant homotopy theory]
    \label{rmk:yeakel}
    Recently, Yeakel and Klang--Yeakel \cite{Yaekel, Klang_Yeakel} developed the homotopy-theoretic foundations of isovariant homotopy theory, and our results can actually be interpreted in their framework. We denote by $\spc_{G, \isov}^\semifree \subseteq \func(\Lambda^2_0, \spc_G)$ the full subcategory of \textit{semifree isovariant space} consisting of spans $X^{f} \leftarrow \partial C \rightarrow C$, where the $G$-action on $X^f$ is assumed to be trivial, while requiring the action on $C$ and $\partial C$ to be free. 
    In particular, semifree isovariant Poincar\'e spaces in the sense of \cref{def:isovariant_structure} form a subgroupoid $\PDsfisov{G} \subseteq \spc_{G, \isov}^{\semifree,\simeq}$.
    Note note that the objects $* \leftarrow G/e \rightarrow G/e$, $* \leftarrow \emptyset \rightarrow \emptyset$ and $\emptyset \leftarrow \emptyset \rightarrow G/e$ generate the category $\spc_{G, \isov}^\semifree$ under colimits, and that mapping out of each of them individually commutes with colimits. Hence, the category of isovariant spaces is equivalent to the category of presheaves on the subcategory spanned by these three objects. This subcategory is equivalent to the subcategory $\sL_G^\semifree \subseteq \sL_G$ of Yeakel's \textit{link orbit category} \cite[Def. 1.1]{Yaekel} spanned by the chains of subgroups $e < G$, $G$ and $e$. Yeakel's homotopy theory of isovariant spaces is (equivalent to) the category of presheaves $\presheaf(\sL_G; \spc)$ on the aforementioned link orbit category, and the discussion above exhibits $\spc_{G, \isov}^\semifree$ as a full subcategory of it. 
    The article \cite{Klang_Yeakel} establishes that the space of isovariant maps between smooth $G$-manifolds, as a subspace of the space of all maps with the compact-open topology, is equivalent to the mapping space in $\presheaf(\sL_G, \spc)$.
\end{rmk}

Extracting the (lower right) pushout corner in \cref{diag:isovariant_structure} provides a map $\PDsfisov{G} \rightarrow \PDsf{G}$.
Given a semifree $G$-Poincar\'e space $X$, we further write $\isovstr(X)$ for the fibre $\isovstr(X) = \PDsfisov{G} \times_{\PDsf{G}} \{X\}$ and similarly for a semifree $G$-Poincar\'e pair $\isovstr(X, \partial X) = \PDsfisovpair{G} \times_{\PDsfpair{G}} \{(X, \partial X)\}$.
Note that the space $\isovstr(X,\partial X)$ is nonempty if and only if $(X,\partial X)$ admits an isovariant structure.

\begin{lem}\label{lem:triviality_of_isovariant_structure}
    Let $X \in \PDsf{G}$.
    Assume that each component of $X^G$ has codimension at least $3$ in the corresponding component of $X^{e}$.
    Then for any (nonequivariant) Poincar\'e pair $(Y, \partial Y)$ there is an equivalence $\map(Y, \isovstr(X)) \simeq \isovstr(X \times (Y, \partial Y))$, natural in arbitrary maps of Poincar\'e pairs $(Y,\partial Y) \to (Z, \partial Z)$.
\end{lem}
As a consequence of this result, observe that, under the codimension $3$ assumption, elements in $\pi_n(\isovstr(X))$ correspond to isovariant structures on $X \times S^n$.
Similarly, a map $S^n \to \isovstr(X)$ is nullhomotopic if and only if the associated isovariant structure on $X \times S^n$ extends to a relative isovariant structure on $X \times (D^{n+1}, S^n)$.
\begin{proof}[Proof of \cref{lem:triviality_of_isovariant_structure}.]
    Consider the spaces
    \begin{align*}
        \A(X) & = \func([1]^2, \spc_G)^\simeq \times_{\func([1] \times 1, \spc_G)^\simeq} \{X^G \to X\};
        \\
        \B(X, \partial X) & = \func([1]^3, \spc_G)^\simeq \times_{\func([1]^2 \times 1, \spc_G)^\simeq} \{(X^G, \partial X^G) \to (X, \partial X) \}
    \end{align*}
    of squares and cubes with prescribed bottom face.
    Straigthening-unstraightening gives an equivalence $\map(Y, \A(X)) \simeq \A(X \times Y)$.
    There is a map $\A(X \times Y) \to \B(X \times (Y, \partial Y))$ by pulling a square with bottom right corner $X \times Y$ back along the map $X \times \partial Y \to X \times Y$.
    We have to show that the composite $\map(Y, \A(X)) \simeq \A(X \times Y) \to \B(X \times (Y, \partial Y))$, which is clearly natural in $(Y, \partial Y) \in \func([1], \spc)$, restricts to the claimed equivalence $\map(Y, \isovstr(X)) \simeq \isovstr(X \times (Y, \partial Y))$.

    First, note that given a commutative cube
    \begin{equation}\label{diag:isovariant_cube}
    \begin{tikzcd}
        & \partial_{01} C \ar[rr] \ar[dd] \ar[dl]
        &
        & \partial_1 C \ar[dd] \ar[dl]
        \\
        \partial_0 C \ar[rr,crossing over] \ar[dd]
        &
        & C
        &
        \\
        & X^G \times \partial Y \ar[rr] \ar[dl]
        & 
        & X \times \partial Y \ar[dl]
        \\
        X^G \times Y \ar[rr]
        &
        & X \times Y \ar[uu,leftarrow,crossing over]
        &        
    \end{tikzcd}
    \end{equation}
    in which the front and back face are pushouts, the left face is a pullback, the top face consists of free $G$-spaces, and the map $\partial_{01} C \to X^G \times \partial Y$ is 2-connected on underlying spaces, then the right face is also a pullback.
    In particular, the whole cube is pulled back from its front face along the map $X \times \partial Y \to X \times Y$.
    By the codimension assumption on $X$, any such cube corresponding to an isovariant structure on $X \times (Y, \partial Y)$ satisfies these conditions and is thus determined by its front face.

    Now suppose that the cube \cref{diag:isovariant_cube} is pulled back from its front face, which is obtained as the unstraightening of a map $Y \to \isovstr(X)$.
    The map $(\partial_0 C, \partial_{01} C) \to (X^G \times Y, X^G \times \partial Y)$ is then a relative equivariant spherical fibration as the unstraightening of a spherical fibration over $Y$.
    The front and back face are also clearly pushout squares and the top face carries a free $G$-action.
    The top face is a Poincar\'e triad as a consequence of \cref{lem:basic_facts_pd_pairs} as the fibres of $(C, \partial_0 C) \to Y$ are Poincar\'e pairs by assumption.
    This shows that the cube gives a relative isovariant structure on $X \times (Y, \partial Y)$.

    Conversely, suppose that the cube \cref{diag:isovariant_cube} defines a relative isovariant structure on $X \times (Y, \partial Y)$.
    The map $\partial_{01}C \to X^G \times \partial Y$ is a spherical fibration whose fibres have dimension at least 3, so it is 2-connected.
    The discussion above shows that the cube is pulled back from its front face.
    It remains to show that the front face is obtained as the unstraightening of a map $Y \to \isovstr(X)$, or equivalently, that the fibre of the front face over each point in $Y$ is an isovariant structure on $X$.
    The map $\partial_0 C \to X^G \times Y$ is an equivariant spherical fibration, so the fibres $\partial_0 C_y \to X^G$ over $y \in Y$ are also equivariant spherical fibrations.
    The fibres $\partial_0 C_y$ and $C_y$ clearly are free $G$-spaces and we just need to show that $(C_y, \partial_0 C_y)$ is a Poincar\'e pair.
    $\partial_0 C_y$ is the total space of a spherical fibration over the Poincar\'e space $X^G$ and thus a compact Poincar\'e space itself.
    Furthermore, the square
    \begin{equation*}
    \begin{tikzcd}
        \partial_0 C^e_y \ar[r] \ar[d]
        & C^e_y \ar[d]
        \\
        X^G \ar[r]
        & X^e
    \end{tikzcd}
    \end{equation*}
    is a pushout square in which the left vertical map is 2-connected and all spaces except for $C^e_y$ are compact, so $C^e_y$ is also compact, see e.g. \cite[Lemma 2.48]{Lueck25}.
    To show that $(C^e_y, \partial_0 C_y)$ is a Poincar\'e pair we now apply the subtraction result from \cref{lem:basic_facts_pd_pairs} to the square above, using that $X^e$ is a Poincar\'e space and $(X^G, \partial_0 C_y)$ a Poincar\'e pair.
    This completes the proof.
\end{proof}

Our goal will be to construct an isovariant structure on a semifree $G$-Poincar\'e space $X$ if the codimension $\dim(X^e)-\dim(X^G)$ is large.
For this, note that a stable variant of the spherical fibration $p \colon \partial C \to X^G$ always exists.

\begin{defn}
    The \textit{stable normal bundle} of a semifree $G$-Poincar\'e space $X$ is the parametrised spectrum
    \begin{equation*}
        \nu_X = \infl D_{X^G} \otimes \equivdualisingspectrum{X}{G}^{-1} \in (\spectra_G)^{X^G}
    \end{equation*}
\end{defn}

It turns out that given an isovariant structure on $X$, the stable normal bundle $\nu_X$ always identifies with a certain stabilisation of $p \colon \partial C \rightarrow X$, which we now recall.
\begin{defn}\label{def:join_stabilisation}
    We define the \textit{join stabilisation} of $G$-spaces as the composite
    \begin{equation*}
        \Sigma^\infty_J \colon \spc_G \xrightarrow{- \join S^0} \spc_{G,*} \xrightarrow{\Sigma^\infty} \spectra_G,
    \end{equation*}
    where the join $X \join S^0$ is the pushout of $* \leftarrow X \rightarrow *$ endowed with the left point as basepoint. Note that it is possible to do this in families, so that one can associate a local system of $G$-spectra to a local system of (unpointed) $G$-spaces over a base.
\end{defn}

\begin{obs}\label{obs:stable_normal_bundle}
    Suppose that we are given an isovariant structure \cref{diag:isovariant_structure} on $X$.
    Then there is an identification $\nu_X \simeq \Sigma_J^\infty p$ of the stable normal bundle $\nu_X$ and the fibrewise join stabilisation of $p$.
    A proof of this uses gluing results for equivariant Poincar\'e pairs from \cite{BHKK_PDP}.
    Let us just give the argument for the underlying nonequivariant spectra, which is sufficient for this article.
    Recall from \cref{thm:basic_properties_equivariant_dualising_spectrum} that there is an equivalence $\res_e^G D_{X, G} \simeq i^* D_{X^e}$, where $i \colon X^G \to X^e$ denotes the inclusion, from which we obtain $\res_e^G \nu_X \simeq D_{X^G} \otimes i^* D_{X^e}^{-1}$.
    Now the claim follows from  \cref{lem:basic_facts_pd_pairs}, which gives us
    \begin{equation*}
        i^* D_{X^e} \simeq D_{(X^{G}, \partial C)} \simeq D_{X^G} \otimes \Sigma^\infty_J p.
    \end{equation*}

    In particular, we get that the fibre of $p^e$ over a point $x \in X^e$ is a $\dim(X^e)-\dim(X^G)-1$-dimensional sphere.
\end{obs}

\begin{strategy}\label{strategy:constructing_isovariant_structure}
The strategy to construct an isovariant structure on $X$ now consists of the following two steps:
\begin{enumerate}
    \item Construct a destabilisation of $\nu_X$, that is a free equivariant spherical fibration $p \colon \partial C \to X$ together with an equivalence $\nu_X \simeq \Sigma^\infty_J X$;
    \item Build the complement $C$ from obstruction theory using Klein's nonequivariant existence result \cref{thm:nonequivariant_existence_embeddings}.
\end{enumerate}
\end{strategy}
These two steps are completely independent.
Step (1) heavily depends on the group $G$ and relies on a good understanding of $\Pic(\spectra_G)$. It is the main content of \cref{sec:destabilisations}.
Step (2) is the content of \cref{sec:existence_isovariant_structures}.
\section{Destabilisations}\label{sec:destabilisations}

This section concerns itself with destabilisations of certain equivariant spherical fibrations, as outlined in the first step of \cref{strategy:constructing_isovariant_structure}.
The ultimate goal is to construct, for a semifre $G$-Poincar\'e space satisfying suitable codimension conditions on the fixed point set, a destabilisation of the stable normal bundle $\nu_X = D_{X^G} \otimes D_{X, G}^{-1} \colon X^G \to \spectra_G$ by finding a lift along the join stabilisation map $\Sigma^\infty_J \colon \hrep_{G} \to \Pic(\spectra_G)$.
In the semifree case, the stable normal bundle $\nu_X$ carries some additional information witnessing that it carries a free $G$-action in a certain sense.
Passing to this finer variant of $\Pic(\spectra_G)$ is crucial to get good connectivity estimates for the stabilisation map $\Sigma^\infty_J$.

\subsection{Generalised homotopy representations and their stabilisations}

Recall from \cref{def:generalised_hrep} the notion of generalised $G$-homotopy representations.
We write $\hrepfree{G} \subseteq \spc_G^{\omega, \simeq}$ for the full subgroupoid of those generalised homotopy representations $X$ which are free, i.e., $X^H = \varnothing$ for $e \neq H \le G$.
The join stabilisation $\Sigma^\infty_J \colon \spc_G \to \spectra_G$ from \cref{def:join_stabilisation} restricts to a map $\Sigma^\infty_J \colon \hrepfree{G} \to \Pic(\spectra_G^\omega)$.

Next, we describe a variant of free invertible $G$-spectra.
Denote by $\spectra^\omega_G/e$ the Verdier quotient by the thick subcategory $\langle G/e \rangle \subseteq \spectra_G^\omega$ generated by $\Sigma^\infty_+ G/e$, i.e., the smallest subcategory containing it closed under finite limits, finite colimits and retracts. 
This happens to be a tensor ideal, so the quotient $\spectra^\omega_G/e$ admits a unique symmetric monoidal structure making the projection $\spectra^\omega_G \rightarrow \spectra^\omega_G/e$ symmetric monoidal.
Let us start with the following observation.
\begin{lem}\label{lem:stabilisation_hrep_mod_e_trivial}
    The composite
    \[ \hrepfree{G} \xrightarrow{\Sigma^\infty_J} \spectra_G^\omega \rightarrow  \spectra_G^\omega/e \]
    is constant with value $\unit$.
\end{lem}
\begin{proof}
    The map factors as the composite
    \begin{equation*}
        \hrepfree{G} \xrightarrow{\join S^0} (\spc_{G, *}^\omega)_{S^0/}^\free \xrightarrow{\Sigma^\infty} (\spectra_G^\omega)_{\sphere/} \to \spectra^\omega_G/e,
    \end{equation*}
    where $(\spc_{G,*}^\omega)_{S^0/}^\free \subseteq (\spc_{G,*}^\omega)_{S^0/}$ is the full subcategory of those $S^0 \to Y$ inducing an equivalence on fixed points for all subgroups $e \neq H \leq G$.
    In particular, the cofibre of the induced map $\Sigma^\infty S^0 \to \Sigma^\infty Y$ lies in the subcategory $\langle G/e \rangle$, showing that $\Sigma^\infty S^0 \to \Sigma^\infty Y$ becomes an equivalence in the quotient $\spectra_G^\omega/e$.
\end{proof}

This motivates the following definition.
\begin{defn}
    A \textit{free invertible $G$-spectrum} is an invertible $G$-spectrum $E \in \spectra^\omega_G$ together with an equivalence $E \simeq \unit$ in the Verdier quotient $\spectra_G^\omega/e$. The moduli space of free invertible $G$-spectra is denoted by 
    \[ \Pic(\spectra_G^\omega)^\free = \Pic(\spectra_G^\omega) \times_{\Pic(\spectra_G^\omega/e)} \{\unit\}. \]
    The \textit{dimension} of a free invertible $G$-spectrum $E$ is $k$, where $k \in \bbZ$ is the degree such that $E^e \simeq \sphere^k$.
\end{defn}

As we have seen above, we can factor $\Sigma^\infty_J$ over a map
\[ \Sigma^\infty_J \colon \hrepfree{G} \rightarrow \Pic(\spectra_G^\omega)^\free \]
which in fact is compatible with the decomposition of both sides according to dimension: letting $\hrepfree{G}(k) \subset \hrepfree{G}$ denote the components of the 
$k-1$-dimensional generalised homotopy representations\footnote{The convention that $\hrepfree{G}(k)$ contains $k-1$ dimensional spheres is made so that these stabilise to the $k$-dimensional sphere spectrum under $\Sigma^\infty_J$, similar to the indexing convention in the definition of the space $G(k) = \mathrm{hAut}(S^{k-1})$ from surgery theory.}, and $\Pic(\spectra_G^\omega)^\free(k) \subset \Pic(\spectra_G^\omega)^\free$ the components of the $k$-dimensional free invertible $G$-spectra, the map restricts to a map $\Sigma^\infty_J \colon \hrepfree{G}(k) \rightarrow \Pic(\spectra_G^\omega)^\free(k)$. 
By \cref{lem:dualising_singular_part_trivial} the stable normal bundle $\nu_X$ also admits a refinement $\nu_X \colon X^G \to \Pic(\spectra^\omega_G)^\free(d^e-d^G)$ if $X$ is semifree, and both $X^e$ and $X^G$ are equidimensional.
The main theorem of this section is the following, which allows us to construct a destabilisation of $\nu_X$.
\begin{thm}
    \label{thm:destabilisations_of_free_hreps}
    Let $G$ be a periodic finite group, and let $k \geq 2$.
    Then the map
    \[ \Sigma^\infty_J \colon \hrepfree{G}(k) \rightarrow \Pic(\spectra_G^\omega)^\free(k) \]
    is $k-1$-connected.
\end{thm}

We recall the notion of a periodic finite group in the next subsection and show that the map is an isomorphism on path components for $k \geq 2$.
The analysis of the higher homotopy takes up the rest of this section, and we prove \cref{thm:destabilisations_of_free_hreps} at the end of \cref{subsec:automorphisms_versus_stable_automorphisms}.

\subsection{Groups with periodic cohomology and free generalised homotopy representations}

Free generalised homotopy representations have been classified by Swan \cite{Swan_periodic}, who clarified their relation to so-called periodic groups. The classification is in terms of classes in Tate cohomology. To stay consistent with the literature, we consider Tate cohomology with the \textit{cohomological grading convention} --- that is, $\widehat{H}^n(G;\bbZ) = \pi_{-n}(\bbZ^{tG})$. For the following result, an \textit{orientation} on a free generalised homotopy representation $X$ of dimension $d$ is an isomorphism $H_d(X;\bbZ) \simeq \bbZ$.

\begin{thm}[Swan]
    \label{thm:swan_classification}
    Let $G \neq 1, C_2$ be a finite group. There is a bijection
    \[ \begin{Bmatrix}
        \text{oriented free generalised $G$-homotopy representations}\\
        \text{of dimension $d$ up to oriented $G$-homotopy equivalence}
    \end{Bmatrix} \xrightarrow{k} \begin{Bmatrix}
        \text{units $t\in \widehat{H}^*(G;\bbZ$)} \\
        \text{of positive degree}
    \end{Bmatrix}.   \]
\end{thm}

\begin{rmk}
    The group $C_2$ is excluded just for convenience of formulation. But let us note that the $d$-sphere with the antipodal action is the unique free $C_2$-homotopy representation for each dimension $d$.
\end{rmk}

\begin{rmk}
    \label{rmk:dimension_of_free_generalised_homotopy_representation}
    Swan's construction in fact shows that every oriented free generalised $G$-homotopy representation of dimension $d$ admits a (possibly infinite) cell structure of dimension $d$. The construction involves an Eilenberg Swindle, see \cite[Lem. 2.22.]{Davis_Milgram_survey}.
\end{rmk}

\begin{cons}
    The map $k$ in \cref{thm:swan_classification} can be constructed as follows. Given a free generalised homotopy representation $X$ of $G$ of dimension $d$, we define its \textit{$k$-invariant} to be the homotopy class of maps
    \begin{align*}
        k(X) = ( \Sigma^\infty S^0 \otimes \bbZ \rightarrow \Sigma^\infty X \join S^0 \otimes \bbZ) \in &\pi_0 \map_{\module_{\bbZ[G]}}(\Sigma^\infty S^0 \otimes \bbZ,\Sigma^\infty X \join S^0 \otimes \bbZ)\\ & \simeq \pi_0 \map_{\module_{\bbZ[G]}}(\bbZ, \bbZ[d+1]) = H^{d+1}(G;\bbZ).
    \end{align*}
    Using that the map $H^{d+1}(G;\bbZ) \rightarrow \widehat{H}^{d+1}(G;\bbZ)$ is an isomorphism, we may view $k(X)$ as an element in Tate cohomology as well. To see that it is in fact a unit, note that $k(X)$ is the image of a map of $G$-spectra which is in fact an equivalence in the stable module category $\stmodSmall_\spectra(G)$, as $X$ carries a free $G$-action.
    Hence it induces a  $\bbZ^{tG}$-linear equivalence $\bbZ^{tG} \simeq (\Sigma^\infty S^0 \otimes \bbZ)^{tG}\rightarrow (\Sigma^\infty X \join S^0 \otimes \bbZ)^{tG} \simeq \bbZ^{tG}[d+1]$. This equivalence is given by multiplication with $k(X)$, so that $k(X)$ has to be a unit.
\end{cons}

\begin{proof}[References for \cref{thm:swan_classification}]
    The result can be extracted from the proof of \cite[Thm. 4.1.]{Swan_periodic}, as mentioned in \cite{Wall_spheres}. A proof is given in \cite[]{Davis_Milgram_survey}.
\end{proof}

\begin{rmk}[Unoriented classification]
    \label{rmk:unoriented_classification}
    To formulate an unoriented version of \cref{thm:swan_classification} that is used later, it is useful to consider the \textit{stable module category} of $G$ with $\bbZ$-coefficients. The category $\func(BG,\module_{\bbZ}^\omega)$ is symmetric monoidal with the pointwise symmetric monoidal structure.
    The stable subcategory generated by $\bbZ[G]$ is a tensor ideal, so that the quotient map
    \[ \func(BG,\module_\bbZ^\omega) \rightarrow \stmodSmall_\bbZ(G) \coloneqq \func(BG,\module^\omega_\bbZ)/\langle \bbZ[G] \rangle \]
    is symmetric monoidal, and the quotient is called the \textit{stable module category} of $G$ with $\bbZ$-coefficients. 
    For $X,Y \in \func(BG,\module_\bbZ^\omega)$, maps in $\stmodSmall_\bbZ(G)$ can be computed by
    \begin{equation}
        \label{eq:maps_in_stmod}
       \map_{\stmodSmall_\bbZ(G)}(X,Y) = \map_{\module_\bbZ}(X,Y)^{tG} 
    \end{equation}
    as shown in \cite[Lem. 4.2.]{krause2020picard}.

    Now, given a generalised homotopy representation $X$, letting $A = H_d(X;\bbZ)$, the construction of $k(X)$ still makes sense as a map $\bbZ \rightarrow A[d+1]$ in $\stmodSmall_\bbZ(G)$. Here, $A$ is considered with the trivial action, reflecting the triviality of the $G$-action on $H_d(X;\bbZ)$ since the action is free and $d$ odd. So we consider the set of tuples $(A,k)$ where $A$ is an infinite cyclic group and $k \colon \bbZ \rightarrow A[d+1]$ is an isomorphism in $\stmodSmall_\bbZ(G)$. Two such tuples $(A,k)$ and $(A',k')$ are called \textit{equivalent} if there is an isomorphism $\alpha \colon A \rightarrow A'$ such that the diagram
    \begin{equation}
        \begin{tikzcd}
            & A[d+1] \ar[dd, "\alpha"] \\
            \bbZ \ar[ur, "k"] \ar[dr, "k'"'] &\\
            & A'[d+1]
        \end{tikzcd}
    \end{equation}
    commutes up to homotopy. A choice of isomorphism $A \cong \bbZ$ identifies $k$ with a unit in $\widehat{H}^{d+1}(G;\bbZ)$ according to the computation of maps in the stable module category \cref{eq:maps_in_stmod}.
\end{rmk}
As a consequence of \cref{thm:swan_classification} we obtain the following unoriented classification of generalised homotopy representations.
\begin{cor}\label{cor:unoriented_classification}
    For $G \neq 1, C_2$ a finite group, the construction above provides an equivalence
     \[ \begin{Bmatrix}
    \text{free generalised $G$-homotopy representations}\\
    \text{of dimension $d$ up to $G$-homotopy equivalence}
\end{Bmatrix} \xrightarrow{(H_d,k)} \begin{Bmatrix}
    \text{tuples $(A,k)$} \\
    \text{up to equivalence}
\end{Bmatrix}.   \]
\end{cor}

The above formulation of the unoriented classification has the advantage that it is very easy to relate it to Krause's stable classification of invertible $G$-spectra.

\begin{cor}
\label{cor:classification_of_free_hreps_up_to_equivalence}
    Let $G$ be a (nontrivial) group and $d \geq 1$.
    The map $\Sigma^\infty_J$ induces a bijection
    \[ \begin{Bmatrix}
        \text{free generalised $G$-homotopy representations}\\
        \text{of dimension $d$ up to $G$-homotopy equivalence}
    \end{Bmatrix} \rightarrow \begin{Bmatrix}
        \text{free invertible $G$-spectra} \\
        \text{of dimension $d$ up to equivalence}
    \end{Bmatrix}.   \]
\end{cor}

\begin{proof}
    For $G = C_2$ both sides consist of a single element for each $d \geq 1$; for the RHS it is written in \cite[Sec. 8.1]{krause2020picard} and for the LHS it is easy to construct a $C_2$-homotopy equivalence out of the sphere with the antipodal $C_2$-action to any free $C_2$-homotopy representation, so we proceed to the case $G \neq C_2$. %By elementary obstruction theory there is some $C_2$-map $f \colon S_\sigma \rightarrow V$ for any free $C_2$-homotopy representation. Since both $C_2 \backslash S_\sigma \rightarrow \mathbb{R}P^d$ and $C_2 \backslash V \rightarrow \mathbb{R}P^d$ are $d$-connected, the map can be computed to have odd degree. Use a pinch map $S_\sigma \rightarrow S^d \wedge S_\sigma \wedge S^d$ where the $C_2$-action switches the wedge summands, which are attached at antipodal points. Note that a $C_2$-map of this is a a $C_2$-map out of $S_\sigma$ (take $f$) and a nonequivariant map out of $S^d$ (take any map with half the difference of the degree of $f$ and $1$. The resulting composite $S_\sigma \rightarrow S^d \wedge S_\sigma \wedge S^d \rightarrow V$ is $C_2$-equivariant and an equivalence on underlying spaces, so an equivariant homotopy equivalence as the action is free.
    In \cite[Thm. 4.16]{krause2020picard}, Krause constructs a $1$-cartesian diagram of spaces as follows. 
    \begin{equation}
    \begin{tikzcd}
        \Pic(\spectra_G^\omega) \ar[r] \ar[d] & \Pic(\spectra_G^\omega/e) \ar[d] \\
        \Pic(\func(BG,\module_\bbZ^\omega)) \ar[r] & \Pic(\stmodSmall_\bbZ(G))
    \end{tikzcd}
    \end{equation}
    This means in particular, that invertible $G$-spectra up to equivalence are determined by an invertible object $L \in \Pic(\spectra^\omega_G/e)$, an invertible object in $A \in \func(BG,\module_\bbZ^\omega)$ and an equivalence $k \colon L \rightarrow A$ in $\stmodSmall_\bbZ(G)$. 
    We can pass to horizontal fibres over $\{\unit\} \rightarrow \Pic(\spectra_G^\omega)$, whose image in $\Pic(\stmodSmall_\bbZ(G))$ is the trivial $\bbZ$-representation, to arrive at the conclusion that elements in $\pi_0\Pic(\spectra^\omega_G)^\free$ are in bijection to the set of tuples $(A,k)$, where $A \in \Pic(\func(BG,\module_\bbZ^\omega))$ and $k \colon \bbZ \rightarrow A$ an equivalence in $\stmodSmall_\bbZ(G)$.
    This identification is set up so that under $\Sigma^\infty_J$ and the unoriented classification \cref{cor:unoriented_classification}, the free generalised homotopy representation corresponding to the datum $(A,k)$ maps to the datum $(A[d+1],k)$, where $A[d+1]$ is the infinite cyclic group $A$ considered in degree $d+1$ with the trivial action. 

    We also note that for free invertible $G$-spectra of dimension $d$, the underlying object in $\func(BG,\module_\bbZ^\omega)$ is concentrated in the degree $d+1$, where it is an infinite cyclic group $A$ with some $G$-action. Our next claim is that this $G$-action is trivial. To see this, note that we have the equivalence $\bbZ \rightarrow A[d+1]$ in $\stmodSmall_\bbZ(G)$. Note that, since for any subgroup $H \leq G$ the restriction of $\bbZ[G]$ splits as a direct sum of copies of $\bbZ[H]$, there is a restriction functor $\stmodSmall_\bbZ(G) \rightarrow \stmodSmall_\bbZ(H)$. In particular, we get induced equivalences $\bbZ \rightarrow \res_H^G A[d+1]$ in $\stmodSmall_\bbZ(H)$ for arbitrary subgroups $H \leq G$, which induce equivalences $\bbZ^{tH} \rightarrow \res_H^G A[d+1]^{tH}$. Since the $G$-action on $A$ is trivial if and only if all its restrictions to cyclic subgroups $C \subset G$ are trivial, we have reduced to the case of a cyclic group. 

   However, if $C$ is cyclic and $A$ carries a nontrivial action (in particular $C$ is nontrivial), then we can compute that $\widehat{H}^*(G;A)$ is concentrated in odd degrees where it is equivalent to $\bbZ/2$. Thus, $\bbZ^{tC} \simeq A[d+1]^{tC}$ can only happen if $d+1$ is odd and $C = C_2$, a case we excluded for this reason.

    All in all, we have seen that $A$ is carries the trivial action. Invoking \cref{rmk:unoriented_classification}, we see that the desired map is indeed a bijection, since both sides are compatibly in bijection to the set of $(A,k)$ where $A$ is an infinite cyclic group and $k \colon \bbZ \rightarrow A[d+1]$ an equivalence in $\stmodSmall_\bbZ(G)$, up to isomorphism.
\end{proof}

A group for which $H^*(G;\bbZ)$ admits a unit in positive (equivalently, nonzero) degree is called a \textit{periodic group}. The classification of periodic groups is classically attributed to Artin and Tate (unpublished), and can be formulated as the following theorem.

\begin{thm}[{\cite[Ch. XII, Sec. 11]{Cartan-Eilenberg}}]
    \label{thm:periodic_groups}
    The following are equivalent for a finite group $G$.
    \begin{enumerate}
        \item Every abelian subgroup of $G$ is cyclic.
        \item For each prime $p$, every $p$-Sylow subgroup of $G$ is either trivial or generalised quaternion.
        \item There is some $n$ such that $H^n(G;\bbZ) \cong \bbZ/\vert G \vert$.
        \item The Tate cohomology ring $\widehat{H}^*(G;\bbZ)$ has a unit in nonzero degree.
    \end{enumerate}
\end{thm}

The set of $n \in \bbZ$ for which there is a unit in $\widehat{H}^n(G;\bbZ)$ is a subgroup, and so generated by a unique positive integer $p$, if $G$ is periodic. This integer $p$ is called the \textit{period} of $G$, and for all multiples of $p$ we have that $\widehat{H}^{kp}(G;\bbZ) \cong \bbZ/\vert G \vert$. The units in $\widehat{H}^{kp}(G;\bbZ)$ are exactly those elements generating it as a cyclic group. The geometric relevance of periodic groups is that if there is a isovariant semifree $G$-Poincar\'e space $X$ for which $X^G$ and $X$ are not equivalent, then $G$ must be periodic.

\subsection{Semifree $G$-spectra}\label{sec:semifree_g_spectra}

Let $G$ be a periodic finite group. 
We would like to study the relation of stable and unstable normal bundles of semifree $G$-Poincar\'e spaces. For this we introduce a custom-made category of $G$-spectra - the category $\spectra_G^{\mathrm{sf}}$ of \textit{semifree $G$-spectra} - which enjoys two desirable properties. First, they form a symmetric monoidal category whose invertible objects are easy to compare to $\Pic(\spectra_G^\omega)^\free$. Second, using equivariant versions of the Blakers-Massey theorem, it is easy to relate maps in $\spectra_G^{\mathrm{sf}}$ to maps between free $G$-spaces.
The construction is a special case of the more general \cref{sec:spectra_specified_isotropy}, and we use them to prove \cref{thm:destabilisations_of_free_hreps}. 

Denote by $\spc_G^\semifree \subseteq \spc_G$ the full subcategory generated by the orbits $G/G$ and $G/e$ under colimits.
The category $\spectra^\semifree_G$ is defined by formally inverting all semifree pointed $G$-homotopy representation in $\spc_{G, *}^\semifree$.
Let us just list the main properties of the category $\spectra^\semifree_G$ that we will need and refer to \cref{sec:spectra_specified_isotropy} for a formal definition and proofs.
We fix a generalised free $G$-homotopy representation $W$ and set $V = S^1 \join W$. 
\begin{enumerate}
    \item There is a symmetric monoidal colimit preserving functor $\Sigma^\infty \colon \spc_{G, *}^\semifree \to \spectra_G^\semifree$ that sends $V$ to a $\otimes$-invertible object.
    \item $\spectra_G^\semifree$ is a stable presentable category and the two orbits $\Sigma^\infty G/e_+$ and $\Sigma^\infty G/G_+$ form a family of compact self-dual generators.
    Consequently, the genuine fixed points 
    \begin{equation*}
        (-)^H \colon \spectra_G^\semifree \to \spectra, \quad X \mapsto \mapsp_{\spectra_G^\semifree}(\Sigma^\infty G/H_+, X)
    \end{equation*}
    for $H = e, G$ are jointly conservative.
    \item There is a symmetric monoidal colimit preserving functor $\spectra_G^\semifree \to \spectra_G$ fitting into a commutative square
    \begin{equation*}
    \begin{tikzcd}
        \spc_{G,*}^\semifree \ar[r, hook] \ar[d, "\Sigma^\infty"]
        & \spc_{G,*} \ar[d, "\Sigma^\infty"]
        \\
        \spectra_G^\semifree \ar[r]
        & \spectra_G.
    \end{tikzcd}
    \end{equation*}
    The geometric fixed points $\Phi^G \colon \spectra_G^\semifree \to \spectra$ are defined as the composite $\spectra_G^\semifree \to \spectra_G \xrightarrow{\Phi^G} \spectra$.
    \item For $Y, Z\in \spc_{G,*}^\semifree$ such that $Y$ is compact the map
    \begin{equation*}
        \colim_n \map_{\spc_{G,*}^\semifree}(V^{\wedge n} \wedge Y, V^{\wedge n} \wedge Z) \xrightarrow{\simeq }\map_{\spectra_G^\semifree}(\Sigma^\infty Y, \Sigma^\infty Z)
    \end{equation*}
    is an equivalence.
\end{enumerate}

The next result allows us to express the mapping spaces appearing in \cref{thm:destabilisations_of_free_hreps} in terms of semifree $G$-spectra.
We again write $(\spectra_G^\mathrm{sf})^\omega / e$ for the Verdier quotient by the thick subcategory of $(\spectra_G^\mathrm{sf})^\omega$ generated by $\Sigma^\infty G/e_+$.

\begin{lem}\label{lem:pullback_semifree}
    In the following diagram, all squares are cartesian.
    \begin{equation}
        \begin{tikzcd}
            (\spectra_G^{\mathrm{sf}})^\omega \ar[r] \ar[d] & (\spectra_G^\mathrm{sf})^\omega / e \ar[d]\\
            (\spectra_G)^\omega \ar[r] \ar[d] &(\spectra_G)^\omega/e \ar[d] \\
            (\spectra^\omega)^{BG} \ar[r]& (\spectra^\omega)^{BG}/e
        \end{tikzcd}
    \end{equation}
\end{lem}

\begin{proof}
    It is shown in \cite[Thm. 3.10]{krause2020picard} that the lower square is cartesian, as a consequence of \cite[Lem. 3.9]{krause2020picard}, and we use the same proof for the outer rectangle.
    This forces the upper square to be cartesian as well.
    
    We again apply \cite[Lem. 3.9]{krause2020picard}.
    First, we show that if $X,Y \in (\spectra_G^\semifree)^\omega$ are such that $X$ is in the thick subcategory generated by $\Sigma^\infty_+ G/e$, then the maps
    \[ \mapsp_{\spectra_G^{\semifree}}(X,Y) \rightarrow \mapsp_{\spectra^{BG}}(X^{e},Y^{e}) \hspace{3mm} \text{and} \hspace{3mm} \mapsp_{\spectra_G^{\semifree}}(Y,X) \rightarrow \mapsp_{\spectra^{BG}}(Y^{e},X^{e})  \]
    are equivalences. Since both $X$ and $Y$ are dualisable, and since $X^\vee$ again lies in the thick subcategory generated by $\Sigma^\infty_+ G/e$ (a consequence of it being self-dual) it suffices to show that the first map is an equivalence.
    Since the statement is stable under colimits and shifts in $Y$ and under finite colimits, shifts and retracts in $X$, it suffices to prove the statement for $X = \Sigma^\infty G/e_+$ and $Y = \Sigma^\infty Z$ for some $Z \in \spc^\semifree_{G,*}$.
    In this case, we use the explicit description of mapping spaces
    \begin{align*}
        \map_{\spectra_G^\semifree}(\Sigma^\infty G/e_+ \wedge S^k, \Sigma^\infty Z) & \simeq \colim_n \map^G_*(G/e_+ \wedge S^k \wedge V^{\wedge n}, Z \wedge V^{\wedge n}) \\
        & \simeq \colim_n \map_*(S^k \wedge (V^e)^{\wedge n}, Z^e \wedge (V^e)^{\wedge n})
        \\
        & \simeq \map_{\spectra}(S^k, Z) 
        \simeq \map_{\spectra^{BG}}(\Sigma^\infty G/e_+ \wedge S^k, Z).
    \end{align*}
    
    It remains to check the second condition of \cite[Lem. 3.9]{krause2020picard} saying that the image of $\langle \Sigma^\infty G/e_+ \rangle \subseteq (\spectra_G^{\mathrm{sf}})^\omega$ in $(\spectra^\omega)^{BG}$ is closed under retracts.
    The restriction functor $\spectra_G^{\mathrm{sf}} \to \spectra^{BG}$ admits a left adjoint sending the generator $\Sigma^\infty G/e_+$ of $\spectra^{BG}$ to $\Sigma^\infty G/e_+ \in \spectra_G^\mathrm{sf}$. 
    It thus maps the thick subcategory of $(\spectra^{\omega})^{BG}$ generated by $\Sigma^\infty G/e_+$ inside $\langle \Sigma^\infty G/e_+ \rangle \subseteq (\spectra_G^{\mathrm{sf}})^\omega$, completing the proof.
\end{proof}

\begin{lem}\label{lem:semifree_mod_e}
    The functor $\Phi^G \colon \spectra_G^\semifree \to \spectra$ induces an equivalence $(\spectra_G^{\mathrm{sf}})^\omega /e \simeq \spectra^\omega$.
\end{lem}
\begin{proof}
    Since $\Phi^G$ is clearly essentially surjective and sends $\Sigma^\infty_+ G/e$ to 0, we only have to show that the induced map $(\spectra_G^{\mathrm{sf}})^\omega/e \to \spectra^\omega$ is fully faithful on the generator $\Sigma^\infty_+ G/G$ of $(\spectra_G^{\mathrm{sf}})^\omega /e$.
    The argument for this is similar to the proof of \cite[Lemma 3.7]{krause2020picard}.
    For $X \in (\spectra^\semifree_G)^\omega$ we compute
    \begin{equation*}
        \mapsp_{(\spectra_G^{\mathrm{sf}})^\omega/e}(\Sigma^\infty G/G_+, X) \simeq \colim_{X \xrightarrow{\sim} X'} \mapsp_{\spectra_G^{\mathrm{sf}}}(\Sigma^\infty G/G_+, X') 
        \simeq (\colim_{X \xrightarrow{\sim} X'} X')^G 
    \end{equation*}
    where the colimit runs over the filtered diagram of all maps $X \to X'$ in $(\spectra^\semifree_G)^\omega$ with cofibre in $\langle G/e \rangle$.
    The same argument as in \cite[Lemma 3.7]{krause2020picard} shows that
    \begin{equation*}
        (\colim_{X \xrightarrow{\sim} X'} X')^e \simeq \mapsp_{\spectra_G^{\mathrm{sf}}}(\Sigma^\infty G/e_+, \colim_{X \xrightarrow{\sim} X'} X') \simeq 0.
    \end{equation*}
    Finally, we can apply \cref{lem:geoemtric_fixed_points_detection} to show
    \begin{equation*}
        (\colim_{X \xrightarrow{\sim} X'} X')^G \simeq \Phi^G (\colim_{X \xrightarrow{\sim} X'} X')
        \simeq \colim_{X \xrightarrow{\sim} X'} \Phi^G X' \simeq \Phi^G X
    \end{equation*}
    using that $\Phi^G X \to \Phi^G X'$ is an equivalence as the cofibre in $\langle G/e \rangle$ has trivial geometric fixed points.
    This completes the proof.
\end{proof}

\subsection{Join stabilisation}

In this section, we study the effect of join-stabilisation on free generalised homotopy representations. Let us first record some elementary facts on joins of objects in a category before specialising to the situation of $G$-spaces of interest.

\begin{recollect}[Joins and slices]
    Suppose that $\sC$ is a category which admits finite limits, finite colimits. 
    The join of two objects $x, y \in \sC$ is defined as $x \join y \coloneqq x \sqcup_{x \times y} y$ and promotes to a functor $- \join s \colon \sC \to \sC_{s/}$.
    If $\sC$ is cartesian closed, with internal hom denoted by $\hom(-,-)$, then  the functor $- \join s$ admits a right adjoint 
    \begin{equation*}
        (-) \join s \colon \sC \rightleftarrows \sC_{s/} \cocolon{\hom_{s/}(*,-)}.
    \end{equation*}
    To see this and give an explicit description, consider the following diagram whose squares are cartesian.
    \begin{equation*}
        \begin{tikzcd}
            \map_{s/}(x \join s,y) \ar[r] \ar[d] & \map(x\join s,y) \ar[r] \ar[d] & \map(x,y) \ar[d] \\
            * \ar[r] & \map(s,y) \ar[r] & \map(x \times s,y) \simeq \map(x,\hom(s,y)).
        \end{tikzcd}
    \end{equation*}
    So indeed, if we set $\hom_{s/}(*,y) \simeq \fib(y \rightarrow \hom(s,y))$ then $\map_{s/}(x \join s,y) \simeq \map(x, \hom_{s/}(*,y))$.
\end{recollect}

We are interested in the connectivity of the map $\map(x,y) \rightarrow \map_{s/}(x\join s,y\join s)$ in the case $\sC = \spc_G$.
It turns out to be easier to instead study the adjunction unit $y \rightarrow \hom_{s/}(*,y\join s)$. 
Let us start with the following nonequivariant result.

\begin{lem}
    \label{lem:compute_connectivity_of_join_unit}
    Let $X$ be a $k$-connected space with $k \geq 0$ and $m\geq 0$. Then the adjunction unit
    \[ X \rightarrow \hom_{S^m/}(*,X \join S^m) \]
    is $2k+1$-connected.
\end{lem}

\begin{proof}
    Assume that $S^m \rightarrow Z$ is constant at $z \in Z$. Then we get a commutative diagram
    \begin{equation}
        \begin{tikzcd}
            \hom_{S^m/}(*,Z) \ar[r] \ar[d] & \map(S^{k+1},Z) \ar[r] \ar[d] & Z \ar[d] \\
            * \ar[r, "z"] & Z \ar[r] & \map(S^k, Z)
        \end{tikzcd}
    \end{equation}
    in which the right and outer rectangle are cartesian.
    Thus, the left square is cartesian which, exhibits an equivalence $\hom_{S^m/}(*,Z) \simeq \Omega^{m+1} Z$. Now if $X$ is nonempty, $S^m \rightarrow X \join S^m$ is nullhomotopic, and a choice of nullhomotopy induces an equivalence $\hom_{S^m/}(*,X \join S^m) \simeq \Omega^{m+1} X \join S^m$. Furthermore, a choice of basepoint provides an identification $X \join S^m \simeq \Sigma^{m+1} X$. Under these identifications, the map
    \[ X \rightarrow \hom_{S^m/}(*,X \join S^m) \simeq \Omega^{m+1} \Sigma^{m+1} X \]
    becomes the usual map, which is $2k+1$-connected.
\end{proof}

Our next goal is to consider the following situation: $X$ is a $d$-dimensional free $G$-homotopy representation, and we want to estimate the connectivity of the map
\[ \map^G(X,X) \rightarrow \map^G_{S^m/}(X \join S^m, X \join S^m). \]
To do so, we recall the following result.
\begin{lem}
    \label{lem:connectivity_on_mapping_spaces}
    Let $f \colon Y \rightarrow Z$ be a map of $G$-spaces, which is $c^e$-connected on underlying spaces and $c^G$-connected on fixed points.
    \begin{enumerate}
        \item For a $G$-CW pair $(X,A)$ so that $(X,A)$ has a free $d^e$-dimensional $G$-CW structure, the map
        \[ \map^G_{A/}(X,Y) \rightarrow \map^G_{A/}(X,Z) \]
        is $d^e-c^e$-connected.
        \item For a semifree $G$-space $X$ the map
        \[ \map^G(X,Y) \rightarrow \map^G(X,Z) \]
        is $\min \{d^e-c^e, d^G-c^G\}$-connected if $X^G$ has a $d^G$-dimensional CW-structure and $(X, X^G)$ has a relative $d^e$-dimensional CW-structure.
    \end{enumerate}
\end{lem}
\begin{proof}
    This follows from elementary equivariant obstruction theory. 
\end{proof}

\begin{cor}\label{cor:join_stabilisation_connectivity}
    Let $X$ be a free $G$-CW complex of dimension $d$ such that $X^{e}$ is $k$-connected. 
    Then the map
    \[ \map^G(X,X) \rightarrow \map^G_{S^m/}(X \join S^m, X \join S^m) \]
    is $2k+1-d$-connected.
\end{cor}
\begin{proof}
    We can assume that $k \ge 0$ as the statement is void otherwise.
    The map in question identifies with the map $\map^G(X,X) \rightarrow \map^G(X, \hom_{S^m/}(*, X \join S^m))$. It is $2k+1-d$-connected by \cref{lem:connectivity_on_mapping_spaces}, using that the map $X \rightarrow \hom_{S^m/}(*,X \join S^m)$ is $2k+1$-connected as seen in \cref{lem:compute_connectivity_of_join_unit}.
\end{proof}

\subsection{Automorphisms versus stable automorphisms of representation spheres}
\label{subsec:automorphisms_versus_stable_automorphisms}
Let $G$ be a finite group and consider a generalised homotopy representation $X \in \spc_{G,*}^\omega$ of the form $X \simeq X^G \join V$ for a free generalised homotopy representation $V$.
Note that $(X,X^G)$ admits a relative CW-structure of dimension $d^{e}$, where $d^e$ is the dimension of the underlying sphere of $X$, since $V$ admits a $G$-CW structure whose dimension agrees with the dimension of $V$ as a sphere. We need a specific adaptation of the equivariant Freudenthal suspension theorem to the world of generalised homotopy representations and semifree $G$-spaces. 
The original proof of the equivariant Freudenthal suspension theorem in \cite{Hauschild} can be modified to yield the following lemma.

\begin{lem}
    \label{lem:semi_free_freudenthal}
    Let $Y$ be a pointed $G$-space and denote by $c^{e}$ and $c^G$ denote the connectivity of $Y^{e}$ and $Y^G$, respectively. Then the adjunction unit map of $G$-spaces
    \[ Y \rightarrow \map_*(X,X \wedge Y) \]
    is $2c^{e}+1$-connected on underlying spaces and $\min\{ 2c^G +1, c^{e}  \}$ connected on fixed points.
\end{lem}

\begin{proof}
    On underlying spaces, the map in question identifies with the adjunction unit $Y \rightarrow \Omega^{d^{e}} \Sigma^{d^{e}} Y$, which is $2c^{e}+1$-connected by the Freudenthal suspension theorem. 
    For the statement about fixed points, rewrite $X = X^G \join V \simeq X^G \join V'$ with $V' = S^0 \join V$ and consider the composition
    \begin{equation}\label{eq:equivariant_freudenthal}
        Y^G \rightarrow \map_*(X^G, X^G \wedge Y^G) \xrightarrow{- \wedge V'} \map_*^G(X,X\wedge Y).
    \end{equation}
    The first map is yet again $2c^{G}+1$-connected. 
    The second map has a section by passing to fixed points, whose fibre is the space $\map_{X^G/}^G(X, X \wedge Y)$.
    The pair $(X, X^{G})$ admits a free $G$-CW-structure of dimension $d^e$ and $(X \wedge Y)^{e}$ is $d^e + c^e + 1$-connected, so $\map^G_{X^G/}(X, X \wedge Y)$ is $c^e + 1$-connected by \cref{lem:connectivity_on_mapping_spaces}. 
    In particular, the map $(-)^G \colon \map_*^G(X,X\wedge Y) \to \map_*(X^G, X^G \wedge Y^G)$ is an isomorphism on homotopy groups in degrees at most $c^e+1$, which implies that the right map in \cref{eq:equivariant_freudenthal} is $c^e$-connected.
    Together, we get that the map $Y^G \to \map_*^G(X,X\wedge Y)$ is $\min\{ 2c^{G} +1, c^{e}  \}$-connected.
\end{proof}

\begin{cor}
    \label{cor:connectivity_of_stabilisation}
    Let $G$ be a periodic group, and $Y,Z$ pointed semifree $G$-spaces such that $Y$ is compact. Write $d^{G}$ for the cellular dimension of $Y^G$, $d^{e}$ for the relative cellular dimension of $(Y,Y^G)$, $c^{G}$ for the connectivity of $Z^{G}$ and $c^{e}$ for the connectivity of $Z^{e}$.
    Then the map
    \[ \map^G_*(Y,Z) \rightarrow \map_{\spectra_G^\mathrm{sf}}(\Sigma^\infty Y, \Sigma^\infty Z) \]
    is $\min\{2c^{e}+1-d^{e},\min\{ 2c^{G}+1,c^{e}\} -d^G \}$ -connected.
\end{cor}

\begin{proof}
    We use the colimit description
    \begin{align*}
        \map_{\spectra^{\mathrm{sf}}_G}(\Sigma^\infty Y ,\Sigma^\infty Z) \simeq \colim_n \map_*^{G}(Y,\map_*(V^{\wedge n},V^{\wedge n}\wedge Z)),
    \end{align*}
    and the map in question is the inclusion of the first component of the filtered colimit diagram.
    The map $\map_*^{G}(Y,Z)  \rightarrow \map_*^{G}(Y,\map_*(V^{\wedge n},V^{\wedge n}\wedge Z))$ is $\min\{2c^{e}+1-d^{e},\min\{ 2c^{G}+1,c^{e}\} -d^G \}$-connected by combining \cref{lem:connectivity_on_mapping_spaces} and \cref{lem:semi_free_freudenthal}.
    This proves the claim.
\end{proof}

\begin{prop}
    \label{prop:connectivity_range_for_automorphisms}
    Let $X$ be a free generalised homotopy representation of the group $G$ of dimension $r \geq 1$. Then the map
    \begin{equation}
        \aut^G(X) \rightarrow \fib(\aut_{\spectra_G^\mathrm{sf}}(\Sigma^\infty_J X) \xrightarrow{\Phi^{G}} \aut_{\spectra_G^{\semifree}/e}(\sphere))
    \end{equation}
    is $r-1$-connected. 
\end{prop}

\begin{proof}
    First, let $l \geq 0$ be an integer. 
    Consider the following diagram.
    \begin{equation}
    \label{diag:stabilisations}
        \begin{tikzcd}
            \aut^{G}(X) \ar[r] \ar[d] & * \ar[d] \\
            \aut^{G}_*(X \join S^l) \ar[r] \ar[d, "f"] & \aut_*(S^l) \ar[d, "g"]\\
            \aut_{\spectra_{G}^{\mathrm{sf}}}(\Sigma^\infty X \join S^l) \ar[r] & \aut_{\spectra}(\Sigma^\infty S^l)
        \end{tikzcd}
    \end{equation}
    Note that $X$ is a free $G$-space admitting a $r$-dimensional cell structure such that $X^e$ is $r-1$-connected. Hence the map
    \[ \aut^G(X) \rightarrow \aut^G_{S^l/}(X \join S^l) \]
    is $2(r-1)+1-r = r-1$-connected by \cref{cor:join_stabilisation_connectivity}. 
    In other words, the upper square in \cref{diag:stabilisations} is $r-1$-cartesian, which is notably independent of $l$.
    The connectivity of the lower right vertical map is linear in $l$. By \cref{cor:connectivity_of_stabilisation}, the connectivity of the the lower left vertical map is, for $l$ large enough, the difference of the connectivity of $X^e\join S^l$ and the cellular dimension of  $S^l$. Hence, the map is $r-l -l = r$-connected. 
    This implies that the lower square is $r-1$-connected. Hence, the outer square is $r-1$-connected as well, being a composite of two $r-1$-connected squares.
    Together with the equivalence $\Phi^G \colon \spectra_G^{\semifree}/e \xrightarrow{\simeq} \spectra$ from  \cref{lem:semifree_mod_e} this proves the claim.
\end{proof}

\begin{proof}[Proof of \cref{thm:destabilisations_of_free_hreps}]
    That the map in question is a bijection on path components is the result of \cref{cor:classification_of_free_hreps_up_to_equivalence}. Given $X \in \hrepfree{G}$ of dimension $k-1$, 
    we apply \cref{lem:pullback_semifree} to indentify the map in the statement of \cref{thm:destabilisations_of_free_hreps} with the map
    \begin{equation}
    \label{eq:connectivity_of_map_on_deloopings}
        B\aut^G(X) \rightarrow \fib(B\aut_{\spectra_G^\mathrm{sf}}(\Sigma^\infty X_J) \xrightarrow{\Phi^{G}} B\aut_{\spectra_G^{\semifree}/e}(\sphere))
    \end{equation}
    obtained by applying the delooping functor $B$ to the map in \cref{prop:connectivity_range_for_automorphisms} and putting $r=k-1$. The delooping functor increases connectivity by $1$, hence the connectivity of \cref{eq:connectivity_of_map_on_deloopings} is $1 + k-2 = k-1$. 
\end{proof}

\section{Constructing complements}\label{sec:existence_isovariant_structures}

In this section we come to the obstruction theoretic part in the proof of our main result about connectivity of the space of isovariant structures on a semifree Poincar\'e space.

\subsection{Complement problems}

Given a destabilisation of the stable normal bundle of the fixed points $X^{C_p} \rightarrow X$, that is a diagram
\begin{equation}
    \begin{tikzcd} 
        \partial C \ar[d, "p"]& \\
        X^{C_p} \ar[r, "\epsilon"] & X
    \end{tikzcd}
\end{equation}
where $p$ is a free spherical fibration stabilising to the stable normal bundle $\nu_X$, we want to complete it to a pushout diagram by finding a suitable \textit{complement} for the embedding filling the upper right corner.
This leads us to the following notation.
\begin{nota}
    A \textit{complement problem} in a category $\sC$ consists of two composable maps as in the left diagram below.
    \begin{equation*}
        \begin{tikzcd}
            x \ar[d, "f"] &    &&  x \ar[r, "g'"] \ar[d, "f"] \ar[dr, phantom, very near end, "\ulcorner"] & x' \ar[d, "f'"] \\
            y \ar[r, "g"] & y' && y \ar[r, "g"] & y'
        \end{tikzcd}
    \end{equation*}
    Given a pushout as on the right in the above diagram, its underlying complement problem is defined to be the diagram on the left, and refer to it as a solution of the complement problem on the left. Note that this depends on the orientation of the pushout, and whenever we write a pushout we choose the down-right direction for its underlying complement problem.
\end{nota}

Before coming to the main result, let us recall some preliminaries needed in the proof.

\begin{recollect}[The Blakers-Massey theorem]
    \label{recollect:blakers_massey}
    Consider a cocartesian square of spaces.
    \begin{equation*}
        \begin{tikzcd}
            A \ar[r, "f"] \ar[d, "g"] & B \ar[d, "g'"]\\
            C \ar[r,"f'"] & D
        \end{tikzcd}
    \end{equation*}
    Then:
    \begin{enumerate}
        \item If $f$ is $n$-connected and $g$ is $m$-connected, then the map $A \rightarrow B \times_C D$ induced by the square is $n+m$-connected;
        \item if $f$ is 2-connected and $g'$ is $n$-connected, then $g$ is $n$-connected.
    \end{enumerate}
    % If $f$ (and hence also $f'$) is $2$-connected and $g'$ is $n$-connected, then $g$ is $n$-connected. Indeed, for $n\leq 1$ this follows from van Kampen's theorem, and for $n\geq 2$ one can use that all maps in questions are equivalences upon taking fundamental groupoids, and hence componentwise pass to universal covers and deduce the statement from excision in homology and the Hurewicz theorem.
\end{recollect}

\begin{recollect}[2-out-of-3 for pushouts]
    \label{recollect:2_out_of_3_for_pushouts}
    Consider a diagram of spaces as follows.
    \begin{equation}
        \begin{tikzcd}
            X \ar[r] \ar[d] & Y \ar[r] \ar[d] & Z \ar[d] \\
            X' \ar[r] & Y' \ar[r] & Z'
        \end{tikzcd}
    \end{equation}
    Then the following 2-out-of-3 properties hold.
    \begin{enumerate}
        \item If the left and right squares are pushouts, then so is the outer rectangle.
        \item If the outer rectangle and the left square are pushouts, then so is the right square.
        \item If the outer rectangle and the right square are pushouts, and moreover the map $Y \rightarrow Z$ induces an equivalence on fundamental groupoids, then also the left square is a pushout.
    \end{enumerate}
\end{recollect}

\begin{lem}
\label{lem:solving_complement_problem_equivariantly_if_given_nonequivariant_solution}
    Consider a complement problem of $G$-spaces as depicted on the left below, and a solution to that complement problem on $H$-fixed points as on the right below.
    \begin{equation}
    \label{eq:complement_problem_and_noneq_solution}
        \begin{tikzcd}
            W \ar[d, "f"] &    &&  W^H \ar[r, "s"] \ar[d, "f^H"] \ar[dr, phantom, very near end, "\ulcorner"] & U \ar[d, "t"] \\
            Y \ar[r, "g"] & Z && Y^H \ar[r, "g^H"] & Z^H
        \end{tikzcd}
    \end{equation}
    Assume that $Z$ is obtained from $Y$ by attaching multiple $k+1$-cells
    along a map $q \colon \coprod_I G/H \times S^k \rightarrow Y$. Assume further that $k \leq \conn(g^H) + \conn(f^H)$, that $\conn(f^H) \geq 2$ and that $t$ is a surjection on path components. Then the equivariant complement problem in \cref{eq:complement_problem_and_noneq_solution} admits a solution, giving the right pushout in \cref{eq:complement_problem_and_noneq_solution} on $H$-fixed points.
\end{lem}

\begin{proof}
    The composite $r \colon \coprod_I S^k \xrightarrow{q^H} Y^H \rightarrow Z^H$ is nullhomotopic since $Z$ is obtained by attaching cells along $q^H$. 
    We can pick a lift $r' \colon \coprod_I S^k \to U$ of $r$ along $t$ as $t$ is surjective on path components. 
    The Blakers-Massey theorem from \cref{recollect:blakers_massey} guarantees that the map $W^H \rightarrow Y^H \times_{Z^H} U$ is $k$-connected, and so we may in particular find a dashed lift $Q^H$ in the following diagram.
    \begin{equation*}      
    \begin{tikzcd}
    \coprod_I S^k \arrow[rdd, "q^H", bend right] \arrow[rrd, "{r'}", bend left] \arrow[rd, "Q^H", dashed] &                  \\      
    & W^H \arrow[d, "f^H"] \arrow[r, "s"] & U \arrow[d, "t"] \\
    & Y^H \arrow[r, "g^H"]                & Z^H             
    \end{tikzcd}
    \end{equation*}
    Now $Q^H$ corresponds to a map $Q \colon \coprod_I G/H \times S^k \rightarrow W$ lifting $q$. We may attach equivariant $k+1$-cells along $Q$ to construct a solution to the original complement problem as follows.
    \begin{equation*}
        \begin{tikzcd}
            W \ar[r] \ar[d] & X = W \coprod_{\coprod_I G/H \times S^k} \coprod_I G/H \ar[d]\\
            Y\ar[r] & Z
        \end{tikzcd}
    \end{equation*}
    To see that the induced map $X^H \rightarrow U$ is an equivalence, 
    apply the third point in \cref{recollect:2_out_of_3_for_pushouts} to the diagram
    \begin{equation*}
        \begin{tikzcd}
            W^H \ar[r, equal] \ar[d] &  W^H \ar[d, "s"] \ar[r, "f^H"] & Y^H \ar[d, "g^H"] \\
            X^H \ar[r] & U \ar[r, "t"] & Z^H.
        \end{tikzcd}
    \end{equation*}
    The map $f^H$ is $2$-connected by assumption and hence an equivalence on fundamental groupoids.
\end{proof}

\begin{obs}[Good cell structures]\label{obs:good_cell_structure}
    Consider a pair $(Z, Y)$ of $G$-spaces such that the map $Y^H \to Z^H$ is an equivalence for all subgroups $e \neq H \le G$.
    Then we can find a relative equivariant CW-structure $(Z_k)_{k}$ on $(Z,Y)$ consisting only of free cells such that the inclusion $Z_k^e \to Z^e$ of the $k$-skeleton is $k$-connected. 
\end{obs}

\begin{lem}\label{thm:solution_equivariant_complement_problem}
    Consider a complement problem of $G$-spaces together with a nonequivariant solution
    \begin{equation}
    \label{eq:next_complement_problem}
        \begin{tikzcd}
            W \ar[d, "f"] &    &&  W^e \ar[r, "s"] \ar[d, "f^e"] \ar[dr, phantom, very near end, "\ulcorner"] & U \ar[d, "t"] \\
            Y \ar[r, "g"] & Z && Y^e \ar[r, "g^e"] & Z^e.
        \end{tikzcd}
    \end{equation}
    Assume the following:
    \begin{enumerate}
        \item $Y^H \to Z^H$ is an equivalence for all $e \neq H \le G$;
        \item $t$ is $0$-connected;
        \item $f^e$ is $2$-connected.
    \end{enumerate}
    Then the complement problem on the left side in \cref{eq:next_complement_problem} admits a solution, giving the right side on underlying spaces.
\end{lem}

\begin{proof}
    We prove the statement by induction over the skeletal filtration $(Z_k)_k$ of $Z$ from \cref{obs:good_cell_structure}.
    We first claim that there are commutative diagrams
    \begin{equation}
    \label{diag:induction_step_cell_attachment}
        \begin{tikzcd}
            W^e \ar[r, dashed] \ar[d, "f^e"]
            & U_k \ar[r, dashed] \ar[d, dashed, "f_{k}"]
            & U_{k+1} \ar[d, dashed, "f_{k+1}"] \ar[r, dashed]
            & U \ar[d, "t"]
            \\
            Y^e \ar[r]
            & Z_k^e \ar[r]
            & Z_{k+1}^e \ar[r]
            & Z^e
        \end{tikzcd}
    \end{equation}
    where all squares are pushouts.
    Note that the map $U_k \to Z_k^e \times_{Z^e} U$ is $k+2$-connected by the Blakers-Massey theorem as $Z_k^e \to Z^e$ is $k$-connected and $f^e$ is 2-connected.
    We can thus lift the attaching maps $\coprod_I S^k \to Z_k^e$ to $U_k$ such that the composite $S^k \to U_k \to U$ is nullhomotopic and define $U_{k+1} = U_k \coprod_{\coprod_I S^k \times G} \coprod_I G$ which fits into a diagram \cref{diag:induction_step_cell_attachment}.
    The middle square is a pushout by construction while the right square is a pushout by 2-out-of-3.
    \cref{lem:solving_complement_problem_equivariantly_if_given_nonequivariant_solution} allows us to inductively extend this to pushouts of $G$-spaces
    \begin{equation*}
    \begin{tikzcd}
            W \ar[d, "f"] \ar[r] 
            &  C_k \ar[d] \ar[r, dashed]
            & C_{k+1} \ar[d, dashed]
            & \\
            Y \ar[r] 
            & Z_k \ar[r]
            & Z_{k+1} \ar[r] 
            & Z
        \end{tikzcd}
    \end{equation*}
    restricting to the left square in \cref{diag:induction_step_cell_attachment} on underlying spaces.
    Taking $C = \colim_k C_k$ gives the desired solution of the complement problem.
\end{proof}

\begin{cor}[Complement problem for pairs]\label{cor:existence_relative_complement_problem}
    Assume we are given a complement problem in $\func([1], \spc_G)$ together with compatible solutions of the boundary problem and the underlying relative problem
    \begin{equation}
    \label{eq:pair_complement_problem}
        \begin{tikzcd}
            (W,\partial W) \ar[d, "f"] 
            &    
            & \partial W \ar[r, "\partial s"] \ar[d, "\partial f"] \ar[dr, phantom, very near end, "\ulcorner"]
            & \partial U \ar[d, "\partial t"] 
            & (W,\partial W) \ar[d, "f^e"] \ar[r, "s"] \ar[dr, phantom, very near end, "\ulcorner"]
            & (V, \partial U^e) \ar[d, "t"]
            \\
            (Y,\partial Y) \ar[r, "g"] 
            & (Z,\partial Z) 
            & \partial Y \ar[r, "\partial g"] 
            & \partial Z
            & (Y^e,\partial Y^e) \ar[r, "g^e"] 
            & (Z^e,\partial Z^e) 
        \end{tikzcd}
    \end{equation}
    Assume that 
    \begin{enumerate}
        \item the map $(Y, \partial Y) \to (Z, \partial Z)$ induces an equivalence on fixed points for all subgroups $e \neq H \le G$;
        \item $t \colon V \to Z^e$ is $0$-connected;
        \item $f^e \colon W^e \to Y^e$ and $\partial f^e \colon \partial W^e \to \partial Y^e$ are 2-connected.
    \end{enumerate}
    Then the equivariant complement problem admits a solution, extending the given solution on the boundary and the nonequivariant solution on underlying spaces.
\end{cor}
\begin{proof}
    Consider the diagram
    \begin{equation}
    \label{eq:boundary_trick}
        \begin{tikzcd}
            W \ar[r] \ar[d] & W \coprod_{\partial W} \partial U \ar[d] & \\
            Y \ar[r] & Y \coprod_{\partial Y} \partial Z \ar[r] & Z
        \end{tikzcd}
    \end{equation}
    and note that the left square is a pushout. 
    Hence, to find a solution of the outer complement problem, we may as well find one for the right complement problem.
    Note that the nonequivariant solution in \cref{eq:pair_complement_problem} gives a solution
    \begin{equation*}
    \begin{tikzcd}
        W^e \coprod_{\partial W^e} \partial U^e \ar[d] \ar[r] & V \ar[d]
        \\
        Y^e \coprod_{\partial Y^e} \partial Z^e \ar[r] & Z^e
    \end{tikzcd}
    \end{equation*}
    to the complement problem \cref{eq:boundary_trick} on underlying spaces.
    Now \cref{thm:solution_equivariant_complement_problem} gives the desired solution to the equivariant complement problem \cref{eq:boundary_trick}.
\end{proof}

\subsection{Proof of the main theorem}

In this section we prove the main result of this article.

\begin{thm}\label{thm:main_thm}
    Let $X$ be a semifree $G$-Poincar\'e space, $G$ a periodic finite group, and let $k \geq -1$ be such that
    \begin{enumerate}
        \item $\dim(X^G) + 3 \le \dim(X^e)$;
        \item $k \le \dim(X^e) - 2 \dim(X^G) - 3 $.
    \end{enumerate}
    Then the space $\isovstr_G(X) = \PDsfisov{G} \times_{\PDsf{G}} \{X\}$ of isovariant structures on $X$ is $k$-connected. 
\end{thm}
\begin{proof}
    We can assume that $X^e$ is connected as $\isovstr_G(X \coprod Y) \simeq \isovstr_G(X) \times \isovstr_G(Y)$.
    Furthermore, we can reduce to the case where $X^{G}$ is nonempty so that $X^{G} \to X^e$ is 0-connected, as the statement is void otherwise.
    Consider a map $f \colon S^n \to \isovstr_G(X)$ for $-1 \le n \le d^e - 2 d^G - 3$.
    By \cref{lem:triviality_of_isovariant_structure} we have to show that the associated isovariant structure
    \begin{equation} \label{diag:main_thm_sn_embedding}
    \begin{tikzcd}
        \partial C \ar[d, "p"] \ar[r]
        & C \ar[d]
        \\
        X^{G} \times S^n \ar[r]
        & X \times S^n
    \end{tikzcd}
    \end{equation}
    on $X \times S^n$, obtained by unstraightening $f$, extends to a relative isovariant structure on $X \times (D^{n+1}, S^n)$. 
    The case $n = -1$ proves the existence of an isovariant structure.

    \textit{1. Existence of an unstable normal bundle:}
    We first argue that the adjoint map $S^n \to \map(X^{G}, \hrepfree{G})$ obtained by straightening of $p$ extends to $D^{n+1}$.
    Note that it becomes constant with value $\nu = i^* D_{X, G}^{-1} \otimes D_{X^{G}}$ after stabilising along $\Sigma^\infty_J \colon \hrepfree{G} \to \Pic(\spectra_G)$.
    Thus, restricted to a fixed component of $X^{G}$, it lands in $\map(X^{G}, \Pic(\spectra_{G}^\omega/e)^\free(d^e - d^G))$, 
    where $d^e = \dim(X^e)$ and $d^G = \dim(X^G)$.
    Now the map $\Sigma^\infty_J \colon \hrepfree{G}(d^e - d^G) \rightarrow \Pic(\spectra_{G}^\omega)^\free(d^e - d^G)$
    is $d^e - d^G -1$-connected by \cref{thm:destabilisations_of_free_hreps}.
    The space $X^G$ is a $d^G$-dimensional Poincar\'e space and thus admits a $d^G$-dimensional cell structure, so the map
    \begin{equation*}
        \map(X^{G}, \hrepfree{G}(d^e - d^G))
        \to \map(X^{G}, \Pic(\spectra_{G}^\omega/e)^\free(d^e - d^G))
    \end{equation*}
    is $d^e - 2 d^G -1$-connected.
    We see that the extension to $D^{n+1}$ exists if $n+1 \le d^e - 2 d^G - 1$.

    \textit{2. Existence of a nonequivariant Poincar\'e embedding:}
    We want to apply Klein's embedding result \cref{thm:nonequivariant_existence_embeddings} to get the existence of a nonequivariant embedding of $X^G \times (D^{n+1}, S^n) \to X^e \times (D^{n+1}, S^n)$ extending the nonequivariant embedding underlying \cref{diag:main_thm_sn_embedding}.
    To check the dimension constraints, note that $X^G \times (D^{n+1}, S^n)$ is a $d^G + n + 1$-dimensional Poincar\'e pair and thus admits a cell structure of that dimension.
    Similarly, $X^e \times (D^{n+1}, S^n)$ is a $d^e + n + 1$-dimensional Poincar\'e pair.
    The nonequivariant embedding exists if $d^G \le d^e - 3$ and $n \le d^e - 2 d^G -3$.

    \textit{3. Identifying spherical fibrations:}
    We want to argue that the relative spherical fibration $q \colon (D, \partial D) \to X^G \times (D^{n+1}, S^n)$  appearing in this embedding agrees with the underlying map $\nu^e$ of the destabilisation constructed in (1).
    For this, note that both are $d^e - d^G -1$-dimensional spherical fibrations which have equivalent stabilisations.
    As $\Sigma^\infty_J \colon \hrep{}(l) \to \Pic(\spectra)(l)$ is $2l-1$-connected by the Freudenthal suspension theorem, and $X^G \times (D^{n+1}, S^n)$ admits a $d^G + n + 1$-dimensional cell structure, both destabilisations are equivalent if $n \le  2 d^e - 3 d^G - 4$. 
    This is implied by $d^e \le d^G - 3$ and $n \le d^e - 2 d^G -3$.

    \textit{4. Extension of the isovariant structure:}
    Now we can apply \cref{cor:existence_relative_complement_problem} to obtain a pushout
    \begin{equation}\label{diag:equiv_pd_embedding_pair}
    \begin{tikzcd}
        (C_1, \partial C) \ar[r] \ar[d, "p_1"]
        & (C, C_2) \ar[d]
        \\
        X^{G} \times (D^{n+1}, S^n) \ar[r]
        & X \times (D^{n+1}, S^n)
    \end{tikzcd}
    \end{equation}
    restricting to the nonequivariant embedding from (2) on underlying spaces.
    In partiuclar, $(C_2; C, C_1; \partial C)$ is a Poincar\'e triad on underlying spaces.
    By construction, $p_1$ is a free equivariant spherical fibration.
    This shows that \cref{diag:equiv_pd_embedding_pair} defines a relative isovariant structure which completes the proof.
\end{proof}

\begin{rmk}\label{rmk:improved_connectivity}
    In the situation of \cref{thm:main_thm}, assume that the map $X^G \to X^e$ is 1-connected.
    The proof shows that the space $\isovstr_G(X)$ is even $k+1$-connected under this assumption.
    One has to use that in step (2), if $X^G \to X^e$ is 1-connected, Klein's result \cref{thm:nonequivariant_existence_embeddings} provides an embedding $X^G \times (D^{n+1}, S^n) \to X^e \times (D^{n+1}, S^n)$ even for $n = k+1$.
\end{rmk}

\appendix
\counterwithin{thm}{section}

\section{Equivariant spectra with specified isotropy}\label{sec:spectra_specified_isotropy}

The goal of this section is to study a variant of the category $\spectra_G$ of $G$-spectra for a finite group $G$, where not all representation spheres but only those with isotropy in a certain collection $\I \subseteq \orbit(G)$ of orbits are invertible.

We always assume $G/G \in \I$ and that for $G/H, G/K \in \I$ every point in the $G$-set $G/H \times G/K$ has isotropy in $\I$. Set $\spc_G^\I = \presheaf(\I) \subset \spc_G$ to be the category of \textit{$G$-spaces with isotropy in $\I$}.
Left Kan extension along the inclusion $b \colon \I \subset \orbit(G)$ identifies $b_! \colon \spc_G^\I \hookrightarrow \spc_G$ as the full subcategory generated under colimits by the orbits $G/H \in \I$.
Since we assume $G/G \in \I$, the category $\spc_G^\I$ has the final object $G/G$, which is in fact a representable presheaf. The condition on products implies that $\spc_{G}^\I \subset \spc_G$ is closed under products, and that the smash product on $\spc_{G,*}$ restricts to $\spc_{G,*}^\I$.

We call $X \in (\spc_G^\I)^\omega$ a \textit{generalised homotopy representation} if $b_! X \in \spc_G^\omega$ is a generalised homotopy representation, i.e., $b_! X^H \simeq S^{n(H)}$ for all subgroups $H \le G$.
The goal of this section is to study basic properties of the formal inversion
\[ \spectra_G^\I \coloneqq \spc_{G,*}^\I[ \{ X \mid X \in (\spc_{G,*}^\I)^\omega \text{ generalised homotopy representation} \}^{-1}]. \]

\begin{recollect}[Formal inversion]\label{rec:formal_inversion}
    Consider a presentably symmetric monoidal category $\sC$ together with a small collection $I \subseteq \sC$ of objects.
    A map $L \colon \sC \to \sC[I^{-1}]$ in $\calg(\presentable^L)$ exhibits $\sC[I^{-1}]$ as the formal inversion of $I$ in $\sC$ if for any $\D \in \calg(\presentable^L)$ the map
    \begin{equation*}
        \func_{\calg(\presentable^L)}(\sC[I^{-1}], \D) \xrightarrow{L^*} \func_{\calg(\presentable^L)}(\sC, \D)
    \end{equation*}
    is the inclusion of the full subcategory on those functors $F \colon \sC \to \D$ sending objects in $I$ to $\otimes$-invertible objects in $\D$.
    The formal inversion always exists by \cite[Section 2.1]{Robalo2014}, see also \cite[Section 6.1]{Hoyois2017}.
    It is shown in \cite[Corollary 2.22]{Robalo2014} that if $x \in \sC$ is $n$-symmetric for some $n \ge 2$, i.e. the cyclic rotation $\sigma \colon x^{\otimes n} \to x^{\otimes n}$ is equivalent to the identity, then the formal inversion $\sC[x^{-1}]$ is given by the telescopic colimit
    \begin{equation}\label{eq:telescopic_colimit_formal_inversion}
        \colim\left(\sC \xrightarrow{x \otimes -} \sC \xrightarrow{x \otimes -} \dots \right) \in \presentable^L
    \end{equation}
    formed in $\presentable^L$.
    Moreover, if $\sC$ is compactly generated and $x\otimes - \colon \sC \to \sC$ preserves compact objects, then $\sC[x^{-1}]$ is compactly generated and we obtain from \cite[Proposition 5.5.7.8]{lurieHTT} and \cite[Lemma 7.23.5.10]{lurieHA}
    \begin{equation}\label{eq:telescopic_colimit_compacts}
        \sC[x^{-1}]^\omega \simeq \colim\left(\sC^\omega \xrightarrow{x \otimes -} \sC^\omega \xrightarrow{x \otimes -} \dots \right) \in \cat,
    \end{equation}
    where the colimit can equivalently be computed in $\cat$ or in $\cat_{\mathrm{rex}}^{\mathrm{idem}}$.
\end{recollect}

The category $\spectra_G$ can be obtained by inverting a single finite dimensional $G$-representation sphere $V$ containing all irreducible ones as a summand.
Similarly, we show that the category $\spectra_G^\I$ of $G$-spectra with isotropy in $\I$ can be obtained as the formal inversion at a single generalised homotopy representation.

\begin{defn}
    \label{def:isotropy_dualising_sphere}
    A generalised $G$-homotopy representation $V \in (\spc_{G,*}^\I)^\omega$ is called an \textit{isotropy dualising sphere} if
    \begin{enumerate}
        \itemsep0em
        \item there exists $V' \in \spc_{G,*}^\I$ such that $V \simeq S^1 \wedge V'$;
        \item the cyclic rotation map $\sigma \colon V^{\wedge n} \rightarrow V^{\wedge n}$ is equivalent to the identity for some $n \ge 2$;
        \item 
        For each $G/H \in \I$, there is an $H$-equivariant
        map $c \colon V \rightarrow (G/H)_+ \wedge V$ such that the composite
        \begin{equation*}
            \res_H^G V \xrightarrow{\res_H^G c} \res_H^G (G/H)_+ \wedge \res_H^G V \xrightarrow{\pi} S^0 \wedge \res_H^G V \simeq \res_H^G V
        \end{equation*}
        is equivalent to the identity, where $\pi$ is induced by the map $G/H \to *$. 
    \end{enumerate}
\end{defn}

\begin{example}\label{ex:iso_separating_sphere_periodic_group}
    Suppose that $G$ is a periodic group and consider $\I = \{G/e, G/G\}$.
    By \cref{sec:destabilisations} there exists a free generalised homotopy representation $W\in \spc_G^\omega$ and we claim that $V = W \join S^1$ is an isotropy separating sphere.
    \begin{enumerate}
        \item We have $W \join S^1 \simeq (W \join S^0) \wedge S^1$.
        \item The cyclic rotation map $\sigma \colon V^{\wedge 3} \rightarrow V^{\wedge 3}$ is equivalent to the identity. 
        Indeed, it has degree one both on fixed points and on underlying spheres. Note that since $W$ has a $G$-CW structure of dimension $d$, there is a $G$-CW structure on $V$ of dimension $d+2$ with fixed CW-space of dimension $1$, and only cells with isotropy $G/e$ and $G/G$.
        In particular, tom Dieck's equivariant Hopf degree theorem \cite[Theorem 8.4.1]{tomDieck_Transformation_and_Repthy} applies to show that $G$-homotopy classes of maps $V \rightarrow V$ are determined by their degrees on fixed points and underlying spaces. This shows that $\sigma$ is equivalent to the identity.
        \item We can lift a nonequivariant Poincar\'e embedding $* \hookrightarrow W/G$ to an equivariant Poincar\'e embedding $G/e \hookrightarrow W$.
        Suspending this further, we obtain an equivariant Poincar\'e embedding 
        \begin{equation}
        \begin{tikzcd}
            S \ar[r] \ar[d] & C \ar[d] \\
            G/e \ar[r, hook] & V.
        \end{tikzcd}
        \end{equation}
        Now consider the map $c \colon V \to G_+ \wedge \res_e^G V$ obtained as the composite
        \begin{equation*}
            V \to \cofib(C \to V) \simeq \cofib(S \to G/e) \simeq G_+ \wedge \cofib(S(e) \to *) \simeq G_+ \wedge \res_e^G V,
        \end{equation*}
        where $S(e)$ denotes the fibre of $S \to G/e$ over $e$.
        This gives the map $c$ with the desired properties.
    \end{enumerate}
\end{example}

For the rest of this section, let us fixed an isotropy separating sphere $V$.
The main result of this section is the following.
\begin{thm}\label{thm:formal_inversion_homotopy_representation}
    Suppose the pair $(G,\I)$ admits an isotropy dualising sphere $V$. 
    Then the symmetric monoidal functor $\spectraisotropysphere{G}{\I}{V} \to \spectra_G^\I$ is an equivalence. 
    Furthermore, the following hold:
    \begin{enumerate}
        \item $\spectra_G^\I$ is a stable category;
        \item the image of the orbits $G/H_+$ under $\Sigma^\infty_\I \colon \spc_{G,*}^\I \rightarrow \spectra_G^\I$ for all $G/H \in \I$ form a family of compact, self-dual generators of $\spectra_G^\I$ under colimits and shifts.
        In particular, the genuine fixed points $X^H = \mapsp_{\spectra_G^\I} (\Sigma^\infty_+ G/H, X)$ are jointly conservative for $G/H \in \I$;
        \item the symmetric monoidal geometric fixed points $\Phi^H \colon \spectra_G^\I \to \spectra$ for $G/H \in \I$  are jointly conservative on compact objects, also see \cref{rem:geometric_fixed_points_conservative}.
    \end{enumerate}
\end{thm}

\begin{cons}
    The geometric fixed points $\Phi^H \colon \spectra_G^\I \to \spectra$ in the previous result are constructed as the symmetric monoidal colimit preserving extension of the composite
    \begin{equation*}
        \spc_{G,*}^\I \hookrightarrow \spc_{G,*} \xrightarrow{(-)^H} \spc_* \xrightarrow{\Sigma^\infty} \spectra,
    \end{equation*}
    which inverts all generalised homotopy representations.
    In particular, there is an equivalence $\Phi^H \Sigma^\infty (-) \simeq \Sigma^\infty (-)^H$.
\end{cons}

For the rest of this section, denote by $\Sigma^\infty_V \colon \spc_{G,*}^\I \to \spectraisotropysphere{G}{\I}{V}$ the formal inversion of $V$ with right adjoint $\Omega^\infty_V \colon \spectraisotropysphere{G}{\I}{V} \to \spc_{G,*}^\I$.
Before proving \cref{thm:formal_inversion_homotopy_representation}, let us establish a few properties of $\spectraisotropysphere{G}{\I}{V}$. The following lemma, a weaker version of part (2) of \cref{thm:formal_inversion_homotopy_representation}, is used in its proof.
\begin{lem}\label{lem:generators_spectra_isotropy}
    The category $\spectraisotropysphere{G}{\I}{V}$ is stable and the orbits $\Sigma^\infty_V V^{\wedge -n} \wedge G/H_+$ for $G/H \in \I$ and $n \ge 0$ form a family of compact generators of $\spectraisotropysphere{G}{\I}{V}$ as a presentable category.
\end{lem}

\begin{proof}
    In $\spc_{G,*}^\I$ we have that $\Sigma(-) \simeq S^1 \wedge -$.
    As $\Sigma^\infty_V$ is symmetric monoidal and colimit preserving, we see that $\Sigma$ is invertible on $\spectraisotropysphere{G}{\I}{V}$ if and only if $\Sigma^\infty_V S^1$ is invertible.
    But this follows from invertibility of $\Sigma^\infty_V V \simeq \Sigma^\infty_V S^1 \otimes \Sigma^\infty_V V'$ using assumption (1) in \cref{def:isotropy_dualising_sphere}.

    Next, recall from \cref{rec:formal_inversion} that, by the cyclic invariance condition, the formal inversion $\spectraisotropysphere{G}{\I}{V}$ is given by the telescopic colimit \cref{eq:telescopic_colimit_formal_inversion}.
    As the orbits $G/H_+$ for $G/H \in \I$ form a family of compact generators of $\spc_{G, *}^\I$, it follows from \cref{eq:telescopic_colimit_formal_inversion} that the objects $\Sigma^\infty_V V^{\wedge -n} \wedge G/H_+$ for $G/H \in \I$ and $n \ge 0$ form a family of compact generators of $\spectraisotropysphere{G}{\I}{V}$.
\end{proof}

Next, we compare $ \spectra_G^\I$ to $\spectra_G$.
Note that the image of $V$ under the colimit preserving symmetric monoidal functor $\spc_{G,*}^\I \xrightarrow{b_!} \spc_{G,*} \xrightarrow{\Sigma^\infty} \spectra_G$ becomes invertible given that $\Sigma^\infty V$ is a compact $G$-spectrum with invertible geometric fixed points for all subgroups $H \le G$.
By construction, this composite factors through a symmetric monoidal colimit preserving functor
\begin{equation*}
    L \colon \spectraisotropysphere{G}{\I}{V} \to \spectra_G.
\end{equation*}

\begin{lem}
    The geometric fixed points $\spectraisotropysphere{G}{\I}{V} \to \spectra_G \xrightarrow{\Phi^H} \spectra$ for all $G/H \in \I$ are jointly conservative on compact objects.
\end{lem}
\begin{proof}
    Consider $E \in (\spectraisotropysphere{G}{\I}{V})^\omega$ with $\Phi^H(E) \simeq 0$ for all $G/H \in \I$.
    If follows from \cref{eq:telescopic_colimit_compacts} that there are $A \in (\spc_{G,*}^I)^\omega$ and $k \ge 0$ together with an equivalence $E \simeq (\Sigma^\infty_V V)^{\otimes -k} \otimes \Sigma^\infty_V A$.
    We compute
    \begin{equation*}
        0 \simeq \Phi^H E \simeq \Phi^H( (\Sigma^\infty V)^{\otimes -k} \otimes \Sigma^\infty A) \simeq \sphere^{-n(h) k} \otimes \Sigma^\infty A^H,
    \end{equation*}
    from which we conclude $\Sigma^\infty A^H \simeq 0$.
    In particular, each component of $A$ is acyclic, i.e., has vanishing homology showing that $S^2 \wedge A^H \simeq 0$.
    Note that this even implies $S^2 \wedge A \simeq 0$ as it is a $G$-space with isotropy in $\I$ all of whose $\I$-fixed points vanish.
    Now $V$ contains $S^1$ as a wedge summand from which we find $V^{\wedge 2} \wedge A \simeq 0$ and consequently
    \begin{equation*}
        E \simeq (\Sigma^\infty_V V)^{\otimes -k-2} \otimes \Sigma^\infty_V (V^{\wedge 2} \wedge A) \simeq 0. \qedhere
    \end{equation*}
\end{proof}

\begin{proof}[Proof of \cref{thm:formal_inversion_homotopy_representation}.]
    Let us start by showing that the orbits $\Sigma^\infty_V G/H_+$ are dualisable and even self-dual for $G/H \in \I$.
    We can construct evaluation and coevaluation maps as follows:
    \begin{align*}
        &\coev \colon V \xrightarrow{c} 
        G/H_+ \wedge V \xrightarrow{\Delta} G/H_+ \wedge G/H_+ \wedge V \\
        &\eval \colon (G/H \times G/H)_+ \wedge V \simeq (G/H_+ \vee T_+) \wedge V \xrightarrow{p} G/H_+ \wedge V \xrightarrow{G/H \to *} V.
    \end{align*}
    The second map is induced by the decomposition of finite $G$-sets $G/H \times G/H \simeq G/H \amalg T$, splitting of the diagonal copy of $G/H$ in $G/H \times G/H$, and $p \colon T_+ \to *$ collapses $T_+$ to the base point.
    As in the proof of the Wirthmüller isomorphism, one checks that the composites
    \begin{align*}
        G/H_+ \wedge V \xrightarrow{\coev} G/H_+ \wedge (G/H_+ \wedge G/H_+ \wedge V) \simeq G/H_+ \wedge G/H_+ \wedge V \wedge G/H_+ \xrightarrow{\eval} V \wedge G/H_+ \\
        V \wedge G/H_+ \xrightarrow{\coev} (G/H_+ \wedge G/H_+ \wedge V) \wedge G/H_+ \simeq G/H_+ \wedge G/H_+ \wedge G/H_+ \wedge V \xrightarrow{\eval} G/H_+ \wedge V
    \end{align*}
    are equivalent to the flip maps, where the middle equivalences swap the third and fourth factor.
    This implies that $\Sigma^\infty G/H_+$ is self-dual in $\spectraisotropysphere{G}{\I}{V}$ as $V$ is invertible in $\spectraisotropysphere{G}{\I}{V}$.
    
    It remains to show that every generalised homotopy representation $Y \in (\spc_{G,*}^\I)^\omega$ is invertible in $\spc_{G,*}^\I[V^{-1}]$.
    This immediately implies that $\spc_{G,*}^\I[V^{-1}] \to \spectra_G^\I$ is an equivalence.
    By compactness, $Y$ lies in the smallest subcategory of $\spc_{G,*}^\I$ which contains the orbits $G/H_+$ for $G/H \in \I$ and is closed under finite colimits and retracts.
    As dualisable objects in a stable category are closed under finite colimits and retracts we see that $\Sigma^\infty_V Y$ is dualisable.
    Invertibility of a dualisable object can be checked after applying the jointly conservative symmetric monoidal geometric fixed point functors, which is clear by $\Phi^H \Sigma^\infty_V Y \simeq \Sigma^\infty Y^H \simeq \sphere^{n(H)}$.

    Finally, let us argue that the orbits $\Sigma^\infty_+ G/H$ for $G/H \in \I$ already generate $\spectraisotropysphere{G}{\I}{V}$.
    By \cref{lem:generators_spectra_isotropy} it suffices to argue that $V^{-n}$ lies in the thick subcategory generated by the orbits under finite limits, finite colimits and retracts.
    But $V^{-n}$ is dual to $V^n$ which belongs to this thick subcategory.
    We showed before that the orbits are self-dual, which implies that the thick subcategory generated by them is also self-dual.
\end{proof}

We will also need an alternative description of geometric fixed points, which generalises the formula $\Phi^G(X) = (X \otimes \widetilde{E\proper})^G$ for $X \in \spectra_G$.
Denote by $E\proper_\I \colon \I\op \to \spc$ the $G$-space with isotropy in $\I$ characterised by 
\begin{equation*}
    E\proper_\I \colon \I\op \to \spc, \quad G/H \mapsto 
    \begin{cases}
        \varnothing & H = G;
        \\
        * & H \neq G.
    \end{cases}
\end{equation*}
The space $\widetilde{E\proper_\I} \in \spc_{G, *}^\I$ is defined by the cofibre sequence $(E \proper_\I)_+ \to S^0 \to \widetilde{E\proper_\I}$.
\begin{prop}\label{prop:geom_fixed_points_eptilde}
    For any $X \in \spectra_G^\I$, the map
    \begin{equation*}
        \Phi^G \colon (X \otimes \widetilde{E\proper_\I})^G \simeq \mapsp_{\spectra_G^\I}(\Sigma^\infty_+ G/G, X \otimes \widetilde{E\proper_\I}) \xrightarrow{\Phi^G} \mapsp_{\spectra}(\sphere, \Phi^G(X)) \simeq \Phi^G(X)
    \end{equation*}
    is an equivalence.
\end{prop}
\begin{proof}
    The proof for this is the same as for $\spectra_G$, see e.g. \cite[Proposition 3.3.8]{Schwede_Global}:
    As $\Sigma^\infty_+ G/G$ is compact, both sides commute with colimits and finite limits in $X$ and it suffices to prove the corresponding statement on mapping spaces for $X = \Sigma^\infty Y$ and $Y \in \spc_{G, *}^\I$.
    The map in question then identifies with the map
    \begin{equation*}
        \colim_n \map_{\spc_{G, *}^\I}(V^{\wedge n}, V^{\wedge n} \wedge X \wedge \widetilde{E\proper_\I}) \xrightarrow{(-)^G} \colim_n \map_{\spc*}((V^G)^{\wedge n}, (V^G)^{\wedge n} \wedge X^G).
    \end{equation*}
    It suffices to show that for any two $A, Z \in \spc_{G, *}^\I$, the map
    \begin{equation}\label{eq:geom_fixed_points_eptilde}
        \map_{\spc_{G, *}^\I}(A, Z \wedge \widetilde{E\proper_\I}) \to \map_{\spc_{G, *}^\I}(A^G, Z \wedge \widetilde{E\proper_\I})
    \end{equation}
    induced by the inclusion $A^G \to A$ is an equivalence.
    This recovers the map obtained by taking fixed points under the identification $\map_{\spc_{G, *}^\I}(A^G, Z \wedge \widetilde{E\proper_\I}) \simeq \map_{\spc_*}(A^G, Z^G)$.
    Now $A$ is obtained from $A^G$ by attaching cells of orbit type $G/H \in \I$ with $H \neq G$.
    As $\widetilde{E\proper_\I}^H \simeq *$, this shows that \cref{eq:geom_fixed_points_eptilde} is an equivalence.
\end{proof}

\begin{lem}\label{lem:geoemtric_fixed_points_detection}
    Suppose that $X \in \spectra_G^\I$ such that $X^H \simeq 0$ for all proper subgroups $H \lneq G$ with $G/H \in \I$.
    Then the map $X^G \to (X\otimes \widetilde{E\proper_\I})^G \simeq \Phi^G(X)$ is an equivalence.
\end{lem}
\begin{proof}
    Equivalently, we can show that the fibre $(X \otimes (E \proper_\I)_+)^G$ is trivial. 
    The argument is an adaption of \cite[Proposition 3.2.19]{Schwede_Global}.
    We prove the more general assertion that for every $A \in \spc_{G, *}^\I$ with $A^G \simeq *$ we have $(X \otimes A)^G \simeq 0$.
    Note that $A$ can be built from $*$ by attaching cells of orbit type $G/H \in \I$ with $H \neq G$.
    The statement follows from induction over this cell structure using that $(X \otimes (G/H_+ \wedge S^n))^G \simeq \Sigma^n \mapsp(G/H_+,X) \simeq \Sigma^n X^H \simeq 0$ by the selfduality of orbits.
\end{proof}

\begin{rmk}\label{rem:geometric_fixed_points_conservative}
    In the case $\I = \{G/G, G/e\}$ when $G$ is periodic, \cref{lem:geoemtric_fixed_points_detection} can be used to show that the geometric fixed points $\Phi^e, \Phi^G \colon \spectra_G^\I \to \spectra$ are jointly conservative.
    A similar argument as in \cite[Proposition 3.3.10]{Schwede_Global} even shows that for general $\I$, the geometric fixed points $\Phi^H \colon \spectra_G^\I \to \spectra$ for $G/H \in \I$ are jointly conservative.
\end{rmk}

\printbibliography

@book{Lueck25,
    author = {W. L\"uck},
    title = {Isomorphism Conjectures in $K$- and $L$-Theory},
    series = {Ergebnisse der Mathematik und ihrer Grenzgebiete},
    publisher={Springer Cham},
    year = {2025}
}

@misc{BHKK_PDP,
      title={Poincar\'e duality pairs of $\infty$-categories}, 
      author={A. Bianchi and K. Hilman and D. Kirstein and C. Kremer},
      year={2025},
      note={\href{https://arxiv.org/abs/2510.20646}{arXiv:2510.20646}},
}

@article {Schultz_Gap,
    AUTHOR = {Schultz, R.},
     TITLE = {Isovariant mappings of degree 1 and the gap hypothesis},
   JOURNAL = {Algebr. Geom. Topol.},
  FJOURNAL = {Algebraic \& Geometric Topology},
    VOLUME = {6},
      YEAR = {2006},
     PAGES = {739--762}
}

@misc{HKK_PD,
      title={Parametrised Poincar\'e duality and equivariant fixed points methods}, 
      author={K. Hilman and D. Kirstein and C. Kremer},
      year={2024},
      note={\href{https://arxiv.org/abs/2405.17641}{arXiv:2405.17641}},
}

@misc{HKK_Nielsen,
      title={Equivariant Poincar\'e duality for cyclic groups of prime order and the Nielsen realisation problem}, 
      author={K. Hilman and D. Kirstein and C. Kremer},
      year={2024},
      note={\href{https://arxiv.org/abs/2409.02220}{arXiv:2409.02220}},
}

@article {Hoyois2017,
    AUTHOR = {Hoyois, M.},
     TITLE = {The six operations in equivariant motivic homotopy theory},
   JOURNAL = {Adv. Math.},
  FJOURNAL = {Advances in Mathematics},
    VOLUME = {305},
      YEAR = {2017},
     PAGES = {197--279},
}

@article {Klein_2001,
    AUTHOR = {Klein, J. R.},
     TITLE = {The dualizing spectrum of a topological group},
   JOURNAL = {Math. Ann.},
  FJOURNAL = {Mathematische Annalen},
    VOLUME = {319},
      YEAR = {2001},
    NUMBER = {3},
     PAGES = {421--456},
}

@misc{kremer2025borelactionsnonpositivelycurved,
      title={Borel actions in nonpositively curved geometry and the Nielsen realisation problem}, 
      author={Christian K.},
      year={2025},
      note={\href{https://arxiv.org/abs/2510.21550}{arXiv:2510.21550}},
}

@article {Klein2002b,
    AUTHOR = {Klein, J. R.},
     TITLE = {Poincar\'e{} duality embeddings and fibrewise homotopy theory.
              {II}},
   JOURNAL = {Q. J. Math.},
  FJOURNAL = {The Quarterly Journal of Mathematics},
    VOLUME = {53},
      YEAR = {2002},
    NUMBER = {3},
}

@book {tomDieck_Transformation_and_Repthy,
    AUTHOR = {tom Dieck, T.},
     TITLE = {Transformation groups and representation theory},
    SERIES = {Lecture Notes in Mathematics},
    VOLUME = {766},
 PUBLISHER = {Springer, Berlin},
      YEAR = {1979},
     PAGES = {viii+309},
}

@article{davis2024nielsen,
      title={On Nielsen realization and manifold models for classifying spaces}, 
      author={J. F. Davis and W. L\"uck},
      year={2024},
      journal={Trans. Amer. Math. Soc.},
      volume={377},
      pages={7557--7600},
}

@misc{krause2020picard,
      title={The Picard group in equivariant homotopy theory via stable module categories}, 
      author={A. Krause},
      year={2020},
      note={\href{https://arxiv.org/abs/2008.05551}{arXiv:2008.05551}},
}

@article {Swan_periodic,
    AUTHOR = {Swan, R. G.},
     TITLE = {Periodic resolutions for finite groups},
   JOURNAL = {Ann. of Math. (2)},
  FJOURNAL = {Annals of Mathematics. Second Series},
    VOLUME = {72},
      YEAR = {1960},
     PAGES = {267--291}
}

@article{Wall_spheres,
    AUTHOR = {Wall, C. T. C.},
     TITLE = {Free actions of finite groups on spheres},
 BOOKTITLE = {Algebraic and geometric topology ({P}roc. {S}ympos. {P}ure
              {M}ath., {S}tanford {U}niv., {S}tanford, {C}alif., 1976),
              {P}art 1},
    SERIES = {Proc. Sympos. Pure Math.},
    VOLUME = {XXXII},
     PAGES = {115--124},
 PUBLISHER = {Amer. Math. Soc., Providence, RI},
      YEAR = {1978}
}

@book {Davis_Milgram_survey,
    AUTHOR = {Davis, J. F. and Milgram, R. J.},
     TITLE = {A survey of the spherical space form problem},
    SERIES = {Mathematical Reports},
    VOLUME = {2, Part 2},
 PUBLISHER = {Harwood Academic Publishers, Chur},
      YEAR = {1985},
     PAGES = {xi+61}
}

@article {Hauschild,
    AUTHOR = {Hauschild, H.},
     TITLE = {\"Aquivariante {H}omotopie. {I}},
   JOURNAL = {Arch. Math. (Basel)},
  FJOURNAL = {Archiv der Mathematik},
    VOLUME = {29},
      YEAR = {1977},
    NUMBER = {2},
     PAGES = {158--165}
}

@book {Cartan-Eilenberg,
    AUTHOR = {Cartan, H. and Eilenberg, S.},
     TITLE = {Homological algebra},
    SERIES = {Princeton Landmarks in Mathematics},
 PUBLISHER = {Princeton University Press, Princeton, NJ},
      YEAR = {1999},
     PAGES = {xvi+390}
}

@article {Robalo2014,
    AUTHOR = {Robalo, M.},
     TITLE = {{$K$}-theory and the bridge from motives to noncommutative
              motives},
   JOURNAL = {Adv. Math.},
  FJOURNAL = {Advances in Mathematics},
    VOLUME = {269},
      YEAR = {2015},
     PAGES = {399--550},
}

@article {WallPD,
    AUTHOR = {Wall, C. T. C.},
     TITLE = {Poincar\'{e} complexes. {I}},
   JOURNAL = {Ann. of Math. (2)},
  FJOURNAL = {Annals of Mathematics. Second Series},
    VOLUME = {86},
      YEAR = {1967},
     PAGES = {213--245}
}

@article{lueck2022brown,
      title={On Brown's Problem, Poincare' models for the classifying spaces for proper actions and Nielsen Realization}, 
      author={W. L\"uck},
      year={2022},
      note={\href{https://arxiv.org/abs/2201.10807}{arXiv:2201.10807}},
}

@book {Schwede_Global,
    AUTHOR = {Schwede, S.},
     TITLE = {Global homotopy theory},
    SERIES = {New Mathematical Monographs},
    VOLUME = {34},
 PUBLISHER = {Cambridge University Press, Cambridge},
      YEAR = {2018},
     PAGES = {xviii+828}
}

@book{lurieHA,
AUTHOR = {Lurie, J.},
TITLE = {Higher algebra},
publisher    = {Harvard University, Cambridge, Massachusetts},
year         = {2017},
note = {Available at \href{https://www.math.ias.edu/~lurie/}{https://www.math.ias.edu/~lurie/}}
}

@book {lurieHTT,
    AUTHOR = {Lurie, J.},
     TITLE = {Higher topos theory},
    SERIES = {Annals of Mathematics Studies},
    VOLUME = {170},
 PUBLISHER = {Princeton University Press, Princeton, NJ},
      YEAR = {2009},
}

@article {Yaekel,
    AUTHOR = {Yeakel, S.},
     TITLE = {An isovariant {E}lmendorf's theorem},
   JOURNAL = {Doc. Math.},
  FJOURNAL = {Documenta Mathematica},
    VOLUME = {27},
      YEAR = {2022},
     PAGES = {613--628}
}

@article {Klang_Yeakel,
    AUTHOR = {Klang, I. and Yeakel, S.},
     TITLE = {Isovariant homotopy theory and fixed point invariants},
   JOURNAL = {Adv. Math.},
  FJOURNAL = {Advances in Mathematics},
    VOLUME = {433},
      YEAR = {2023},
}

@article {Madsen_Thomas_Wall,
    AUTHOR = {Madsen, I. and Thomas, C. B. and Wall, C. T. C.},
     TITLE = {The topological spherical space form problem. {II}.
              {E}xistence of free actions},
   JOURNAL = {Topology},
  FJOURNAL = {Topology. An International Journal of Mathematics},
    VOLUME = {15},
      YEAR = {1976},
    NUMBER = {4},
     PAGES = {375--382}
}
\end{document}